\documentclass{amsart}

\usepackage{amsmath, amssymb}
\usepackage{amscd}
\usepackage{verbatim}
\newtheorem{definition}{Definition}[section]
\newtheorem{theorem}[definition]{Theorem}
\newtheorem{lemma}[definition]{Lemma}
\newtheorem{corollary}[definition]{Corollary}
\newtheorem{proposition}[definition]{Proposition}
\theoremstyle{definition}
\newtheorem{remark}[definition]{Remark}

\newtheorem{example}[definition]{Example}

\newtheorem{notation}[definition]{Notation}
\newtheorem{question}[definition]{Question}

\newcommand\style{\mathcal }          

\newcommand\CC{\mathbb{C}}

\newcommand{\R}{\mathbb{R}}

\newcommand{\cl}[1]{\mathcal{#1}}
\newcommand{\bb}[1]{\mathbb{#1}}

\newcommand{\B}{\style{B}}
\newcommand{\M}{\style{M}}
\renewcommand{\H}{\style{H}}
\newcommand{\K}{\style K}
\newcommand{\E}{\style{E}}

\newcommand\fn{{\mathbb F}_n}         
\newcommand\fm{{\mathbb F}_m}         
\newcommand\fk{{\mathbb F}_k}         

\newcommand{\lgG}{\rm{G}}        
\newcommand{\lgH}{\rm{H}}        

\newcommand\osr{{\style R}}

\newcommand\oss{{\style S}}
\newcommand\ost{{\style T}}

\newcommand\omin{\otimes_{\rm min}}
\newcommand\omax{\otimes_{\rm max}}
\newcommand\oc{\otimes_{\rm c}}
\newcommand\oess{\otimes_{\rm ess}}

\newcommand\coisubset{\subseteq_{\rm coi}}   

\newcommand\jay{{\style J}}                                  
\newcommand\cstar{{\rm C}^*}                              
\newcommand\cstare{{\rm C}_{\rm e}^*}              
\newcommand\cstaru{{\rm C}_{\rm u}^*}              
\newcommand\wstar{{\rm W}^*}                              

\begin{document}

\title[Operator Systems from Discrete Groups]{Operator
  Systems from Discrete Groups}

\author[D.~Farenick]{Douglas Farenick}
\address{Department of Mathematics and Statistics, University of Regina,
Regina, Saskatchewan S4S 0A2, Canada}
\email{douglas.farenick@uregina.ca}

\author[A.~S.~Kavruk]{Ali S.~Kavruk}
\address{Department of Mathematics, University of Illinois, Urbana, IL 61801, U.S.A.}
\email{kavruk@illinois.edu}

\author[V.~I.~Paulsen]{Vern I.~Paulsen}
\address{Department of Mathematics, University of Houston,
Houston, Texas 77204-3476, U.S.A.}
\email{vern@math.uh.edu}

\author[I.~G.~Todorov]{Ivan G.~Todorov}
\address{Pure Mathematics Research Centre, Queen's University Belfast, Belfast BT7 1NN, United Kingdom}
\email{i.todorov@qub.ac.uk}

\thanks{This work supported in part by NSERC (Canada), NSF (USA), and EPSRC (United Kingdom)}

\begin{abstract}
We formulate a general framework for the study of operator systems arising from discrete groups. We study in detail
the operator system of the free group $\fn$ on $n$ generators, as well as the operator systems of
the free products of finitely many copies of the two-element group $\mathbb Z_2$.
We examine various tensor products of group operator systems, including the minimal, the maximal, and the commuting
tensor products. We introduce a new tensor product in the category of operator systems and
formulate necessary and sufficient conditions for its equality to the commuting tensor product in the case of group operator systems.
We express various sets of quantum correlations studied in the theoretical physics literature
in terms of different tensor products of operator systems of discrete groups.
We thus recover earlier results of Tsirelson and formulate a new
approach for the study of quantum correlation boxes.
\end{abstract}

\maketitle

\section*{Introduction}

In this paper we study operator systems arising from discrete groups.
Let $\mathfrak u$ be a set of generators of a discrete group $\lgG$. The
operator system of $\mathfrak u$ is defined to be
the operator subsystem $\oss(\mathfrak u)$ of the (full) group C$^*$-algebra $\cstar(\lgG)$ given by
\begin{equation}\label{osdg}
\oss(\mathfrak u)\,=\,\mbox{\rm span}\{1,u,u^*\,:\,u\in\mathfrak u\}\,\subseteq\,\cstar(\lgG)\,.
\end{equation}
We will show that many ``universal operator systems''
arise in this manner. For example, the universal operator systems of
$n$ unitary operators or of $n$ contractive operators arise from the study of
the operator system $\oss(\mathfrak u)$ obtained by letting $\mathfrak u = \{u_1,\dots,u_n\}$ be the
set of standard generators of the free group $\mathbb F_n$ on $n$ generators (see Section \ref{op th}).

We show that a number of questions studied previously in the literature can be placed in the framework of
tensor products of group operator systems.
For example, in our earlier work \cite{farenick--paulsen2011,farenick--kavruk--paulsen2011,kavruk2011} it has been established that
Connes' Embedding Problem, in the guise of Kirchberg's Problem,
reduces to understanding the ``minimal'' and the ``commuting'' tensor products \cite{kavruk--paulsen--todorov--tomforde2011}
of some group operator systems. Specifically, if we let $\oss_n$ denote the operator
subsystem of $\cstar(\fn)$ obtained as in \eqref{osdg} using the set $\mathfrak u$ of
the canonical generators for the free group $\fn$, then the following statements are equivalent:
\begin{enumerate}
\item[{(i)}] $\oss_2\omin\oss_2=\oss_2\oc\oss_2$;
\item[{(ii)}] $\cstar(\mathbb F_2)\omin\cstar(\mathbb F_2) = \cstar(\mathbb F_2)\omax\cstar(\mathbb F_2)$;
\item[{(iii)}] the Connes Embedding Problem has an affirmative solution.
\end{enumerate}
The equivalence of the last two assertions above is the criterion of Kirchberg \cite{kirchberg1993}; however, statement (i)
involves a vector space of dimension $9$ and is, on the surface, seemingly more tractable than the Kirchberg criterion.

Because of their ``universal'' and ``free'' properties, 
there are certain tensor identities that hold for tensor products of the form $\oss_n \otimes \oss_m$ that are not apparent for other group operator systems. 
For this reason we are forced to introduce a new tensor product, which we denote by ${\rm ess}$.
It complements the list of tensor product introduced in \cite{kavruk--paulsen--todorov--tomforde2011}
and plays an important role in the study of group operator systems. In Section \ref{s_osdg} we determine
necessary and sufficient conditions for the equality of the ess and the commuting tensor products in the case of group operator systems.

The problem of whether $\oss_2\omin\oss_2=\oss_2\oc\oss_2$ is equivalent to the Connes Embedding Problem. What is the relationship between $\oss_2\oc\oss_2$ and $\oss_2\omax\oss_2$?
In Theorem \ref{ocomax} we show that $\oss_n\oc\oss_m\neq\oss_n\omax\oss_m$ for every $n,m\in\mathbb N$ and that the distinction between them occurs at the level of $2\times 2$ matrices
over $\oss_n\oc\oss_m$ and $\oss_n\omax\oss_m$. One interesting consequence: although the C$^*$-envelope $\cstare(\oss_k)$ of $\oss_k$ is 
$\cstar(\mathbb F_k)$, it is not true that $\cstare(\oss_n\omax\oss_m)=\cstar(\fn)\omax\cstar(\fm)$ (Theorem \ref{omaxe}).

The study of $\oss_n$ in the case where $n=1$ is of interest
in operator theory, as the operator system $\oss_1$ and its dual $\oss_1^d$ are universal objects for operators of a certain type, as we shall explain in 
Section \ref{op th}. A complete list of tensor product relations between $\oss_1$ and $\oss_1^d$ is given in Theorem \ref{mixed}
for the min, c, and max tensor products.

There are further reasons for undertaking a detailed study of tensor products of operator systems arising from discrete groups.
As suggested, for example, in \cite{fritz2012},
operator system tensor products offer a useful framework for the
description of quantum correlations. However, in \cite{fritz2012} only
the ordered structure is used fully and the matrix orderings are not
exploited or characterized.
In the present paper, we make explicit the connections between quantum correlations and
tensor products of group operator systems.
In this regard, we show in Section \ref{tsir}  that the
``non-commutative $n$-cube'' studied by Tsirelson \cite{tsirelson1980,tsirelson1993} in connection with quantum generalisations of the Bell inequalities
generates an operator system, which we denote by $NC(n)$, that is also
the operator system arising from the $n$-fold free product $\mathbb Z_2*\cdots*\mathbb Z_2$.
In Theorem \ref{t-thm} we establish an operator system interpretation of Tsirelson's computations in \cite{tsirelson1980,tsirelson1993}: namely,
$NC(n)\oc NC(m)\neq NC(n)\omax NC(m)$ for all $n,m\geq 2$.

Finally, in Section \ref{bcb} we examine bipartite correlations from the perspective of three operator system structures on the
algebraic tensor product $NC(2)^d\otimes NC(2)^d$, where $NC(2)^d$ is given a concrete realisation as an operator subsystem
of $4\times 4$ diagonal matrices.


\section{Tensor Products and Quotients of Operator Systems}

In this section, we introduce basic terminology and notation, and recall previous constructions and results
that will be needed in the sequel.
If $V$ is a vector space, we let $\cl M_{n,m}(V)$ be the space of all $n$ by $m$ matrices with entries in $V$.
We set $\cl M_n(V) = \cl M_{n,n}(V)$ and $\cl M_n = \cl M_n(\mathbb{C})$.
We let $(E_{ij})_{i,j}$ be the canonical matrix unit system in $\cl M_n$.
For a map $\phi : V\to W$ between vector spaces, we let $\phi^{(n)} : \cl M_n(V)\to \cl M_n(W)$
be the $n$th ampliation of $\phi$ given by $\phi^{(n)}((x_{ij})_{i,j}) = (\phi(x_{i,j}))_{i,j}$.
For a Hilbert space $\cl H$, we denote by $\cl B(\cl H)$ the
algebra of all bounded linear operators on $\cl H$.
An \emph{operator system} is a subspace $\cl S$ of
$\cl B(\cl H)$ for some Hilbert space $\cl H$ which contains the identity operator $I$ and is closed
under taking adjoints. The embedding of $\cl M_n(\cl S)$
into $\cl B(\cl H^n)$ gives rise to the cone $\cl M_n(\cl S)_+$ of all positive operators in $\cl M_n(\cl S)$. The
family $(\cl M_n(\cl S)_+)_{n\in \mathbb{N}}$ of cones is called the \emph{operator system structure} of $\cl S$.
Every complex $*$-vector space equipped with a family of matricial cones and an order unit satisfying natural axioms can, by virtue
of the Choi-Effros Theorem \cite{choi--effros1977}, be represented faithfully as an operator system acting on some Hilbert space.
When a particular embedding is not specified, the order unit of an operator system will be denoted by $1$.
A map $\phi : \cl S\to \cl T$ between operator systems is called \emph{completely positive} if $\phi^{(n)}$
positive, that is, $\phi^{(n)}(\cl M_n(\cl S)_+)\subseteq \cl M_n(\cl T)_+$, for every $n\in \mathbb{N}$.
A linear bijection $\phi : \cl S\to \cl T$ of operator systems $\oss$ and $\ost$ is a \emph{complete order isomorphism} if both $\phi$ and $\phi^{-1}$
are completely positive.
We refer the reader to
\cite{Paulsen-book} for further properties of operator systems and completely positive maps.

An \emph{operator system tensor product} $\oss\otimes_\tau\ost$ of operator systems $\oss$ and $\ost$
is an operator system structure on the algebraic tensor product $\oss\otimes\ost$ satisfying a set of natural axioms.
We refer the reader to \cite{kavruk--paulsen--todorov--tomforde2011}, where a detailed study of such tensor products
was undertaken.
Suppose that $\oss_1\subseteq\ost_1$ and $\oss_2\subseteq\ost_2$ are inclusions of operator systems. Let $\iota_j:\oss_j\rightarrow\ost_j$ denote the
inclusion maps $\iota_j(x_j)=x_j$ for $x_j\in\oss_j$, $j = 1,2$, so that the map
$\iota_1\otimes\iota_2:\oss_1\otimes\oss_2\rightarrow\ost_1\otimes\ost_2$ is a linear
inclusion of vector spaces. If $\tau$ and $\sigma$ are operator system structures on $\oss_1\otimes\oss_2$ and $\ost_1\otimes\ost_2$ respectively,
then we use the notation
\[
\oss_1\otimes_\tau\oss_2\,\subseteq^+\,\ost_1\otimes_\sigma\ost_2
\]
to denote that $\iota_1\otimes\iota_2:\oss_1\otimes_\tau\oss_2\rightarrow\ost_1\otimes_\sigma\ost_2$ is a
(unital) completely positive map. This notation is motivated by the fact that $\iota_1 \otimes \iota_2$ is a
completely positive map if and only if, for every $n,$ the cone $\cl M_n(\oss_1\otimes_\tau\oss_2)_+$ is
contained in the cone $\cl M_n(\ost_1\otimes_\sigma \ost_2)_+$.
If, in addition, $\iota_1\otimes\iota_2$ is a complete order isomorphism onto its range, then we write
\[
\oss_1\otimes_\tau\oss_2\coisubset\ost_1\otimes_\sigma\ost_2\,.
\]
In particular, if $\tau$ and $\sigma$ are two operator system structures on $\oss\otimes\ost$, then
\[
\oss\otimes_\tau\ost\,=\,\oss\otimes_\sigma\ost
\quad\mbox{means}\quad
\oss\otimes_\tau\ost\coisubset\oss\otimes_\sigma\ost \quad\mbox{and}\quad
\oss\otimes_\sigma\ost\coisubset\oss\otimes_\tau\ost \,.
\]

When $\oss_1\otimes_\tau\oss_2 \subseteq^+\oss_1\otimes_\sigma \oss_2,$ then we will also write $\tau \ge \sigma$
and say that $\tau$ majorizes $\sigma$.

In the sequel, we will use extensively the following operator system tensor products
introduced in \cite{kavruk--paulsen--todorov--tomforde2011}:

\smallskip

(a) \emph{The minimal tensor product $\min$.} If $\cl S\subseteq \cl B(\cl H)$ and $\cl T\subseteq \cl B(\cl K)$, where
$\cl H$ and $\cl K$ are Hilbert spaces, then $\cl S\otimes_{\min}\cl T$ is the operator system arising from the
natural inclusion of $\cl S\otimes\cl T$ into $\cl B(\cl H\otimes\cl K)$.

\smallskip

(b) \emph{The maximal tensor product $\max$.} For each $n\in \mathbb{N}$, let
$D_n = \{A^*(P\otimes Q)A : A\in \cl M_{n,km}(\mathbb{C}), P\in \cl M_k(\cl S)_+, Q\in \cl M_m(\cl T)_+\}$.
The Archimedanization \cite{paulsen--todorov--tomforde2011} of the family $(D_n)_{n\in \mathbb{N}}$
of cones is an operator system structure on $\cl S\otimes\cl T$; the corresponding operator system
is denoted by $\cl S\otimes_{\max}\cl T$.

\smallskip

(c) \emph{The commuting tensor product ${\rm c}$.} By definition, an element $X\in \cl M_n(\cl S\otimes\cl T)$
belongs to the postive cone $\cl M_n(\cl S\oc\cl T)_+$ if
$(\phi\cdot\psi)^{(n)}(X)$ is a positive operator for all completely positive maps $\phi : \cl S\to \cl B(\cl H)$
and $\psi : \cl T\to \cl B(\cl H)$ with commuting ranges. Here, the linear map
$\phi\cdot\psi : \cl S\otimes \cl T\to \cl B(\cl H)$
is given by $\phi\cdot\psi(x\otimes y) = \phi(x)\psi(y)$, $x\in \cl S$, $y\in \cl T$.

\smallskip

The tensor products min, c, and max are functorial in the sense that
if $\tau$ denotes any of them, and $\phi : \cl S_1\to \cl S_2$ and $\psi : \cl T_1\to \cl T_2$ are
completely positive maps, then the tensor product map
$\phi\otimes\psi : \cl S_1\otimes_{\tau}\cl T_1\to \cl S_2\otimes_{\tau}\cl T_2$ is completely positive.

A new tensor product ${\rm ess}$ will be introduced in the next section. This tensor product and the three tensor products
mentioned above satisfy the relations
\[
\oss\omax\ost\,\subseteq^+\,\oss\oc\ost\,\subseteq^+\,\oss\oess \ost \subseteq^+ \oss\omin\ost
\]
for all operator systems $\oss$ and $\ost$.

For every operator system $\cl S$, we denote by $\cl S^d$ the (normed space) dual of $\cl S$.
The space $\cl M_n(\cl S^d)$ can be naturally identified with a subspace of the space
$\cl L(\cl S,\cl M_n)$ of all linear maps from $\cl S$ into $\cl M_n$. 
Taking the pre-image of the cone of all completely positive maps in $\cl L(\cl S,\cl M_n)$, we obtain a family 
of matricial cones on $\cl S^d$, which turns it into an operator system. 
We have, in particular, that $\cl S^d_+$ consists of all positive functionals on $\cl S$; the elements $\phi\in \cl S^d_+$ with $\phi(1) = 1$
are called \emph{states} of $\cl S$. An important case arises when $\cl S$ is finite dimensional; in this case, $\cl S^d$ is an operator system
when equipped with the matricial cones family just described and an order unit is given by any faithful state on $\cl S$ \cite[Corollary 4.5]{choi--effros1977}.

We now move to the notion of quotients in the operator system category.

\begin{definition} A linear subspace $\jay\subseteq \oss$ of an operator system $\oss$ is called a \emph{kernel} if there is an operator system $\ost$
and a completely positive linear map $\phi:\oss\rightarrow\ost$ such that $\jay=\ker\phi$.
\end{definition}

If $\jay\subseteq\oss$ is kernel, then one may endow the $*$-vector space $\oss/\jay$ with an operator system structure
such that the canonical quotient map $q_\jay:\oss\rightarrow\oss/\jay$ is unital and completely positive
\cite{kavruk--paulsen--todorov--tomforde2010}.
Moreover, if $\jay\subseteq\ker\phi$ for some completely positive map $\phi:\oss\rightarrow\ost$,
then there exists a completely positive map $\dot{\phi}:\oss/\jay\rightarrow\ost$
such that $\phi=\dot{\phi}\circ q_\jay$.

\begin{definition} A completely positive map $\phi:\oss\rightarrow\ost$ is called a
\emph{complete quotient map} if the natural quotient map $\dot{\phi}:\oss/\ker\phi\rightarrow\ost$
is a complete order isomorphism.
\end{definition}

The following complete quotient map will be used in the sequel.

\begin{example}\label{quo map} {\rm (\cite{farenick--paulsen2011})}
Let $u_1,\dots,u_{k-1}$ denote free unitary generators of $\cstar(\mathbb F_{k-1})$ and set $\oss_{k-1}\subset \cstar(\mathbb F_{k-1})$ to be the span of
$\{1, u_1, u_1^*,\dots,u_{k-1}, u_{k-1}^*\}$.
Let $\ost_k$ operator subsystem of $\M_k$ consisting of all tri-diagonal matrices and let
$\{E_{ij}\}_{|i-j|\leq 1}\subset\M_k$ denote the matrix units that span $\ost_k$. Then the linear map
$\phi_k:\ost_k\rightarrow\oss_{k-1}$ for which
\[
\phi_k(E_{ii}) = \frac{1}{k}1,\;  \phi_k(E_{j,j+1}) = \frac{1}{k}u_j,
\;\mbox{ and }\; \phi_k(E_{j+1,j}) = \frac{1}{k}u_{j}^*
\]
for $i=1,\dots,k$ and $j=1,\dots,k-1$, is a complete quotient map.
\end{example}

We remark that the operator system quotient $\ost_k/\ker\phi_k$ in Example \ref{quo map} is in fact different
from the operator space quotient $(\ost_k/\ker\phi_k)_{{\rm osp}}$ within the category of operator spaces
\cite{farenick--paulsen2011}.

To determine which surjective completely positive maps are complete quotient maps, there is a useful criterion based on strict positivity.

\begin{definition}\label{strict pos} An element $h$ of an operator system $\oss$ is called
\emph{strictly positive} if there is a $\delta>0$ such that $h-\delta 1_\oss\in\oss_+$.
\end{definition}

The following fact will be useful for us on several occasions.

\begin{proposition}\label{quo map criterion} {\rm (\cite[Proposition 3.2]{farenick--kavruk--paulsen2011})}
A completely positive surjection $\phi:\oss\rightarrow\ost$ is a complete quotient map if and only if
for every $p\in\mathbb N$ and every strictly positive element $Y\in\M_p(\ost)$ there is a strictly positive
element $X\in\M_p(\oss)$ such that $Y=\phi^{(p)}(X)$.
\end{proposition}

We record here for future reference two facts about positivity, the first of which is elementary and its proof is thus omitted.

\begin{lemma}\label{schur} Let $\ost$ be an operator system.
If $A=[\alpha_{ij}]_{i,j}\in\M_p(\CC)_+$ and $T=[t_{ij}]_{i,j}\in\M_p(\ost)_+$,
then $A\circ T:=[\alpha_{ij}t_{ij}]_{i,j}\in\M_p(\ost)_+$.
\end{lemma}

\begin{lemma}\label{f-dim max} If $h\in\oss\omax\ost$ is strictly positive, then
there exist $n\in\mathbb N$, $S=[s_{ij}]\in\M_n(\oss)_+$ and $T=[t_{ij}]\in\M_n(\ost)_+$ such that
\begin{equation}\label{expression}
h\,=\,\sum_{i,j=1}^n s_{ij}\otimes t_{ij}\,.
\end{equation}
\end{lemma}

\begin{proof} By hypothesis, $h-\delta1\in(\oss\omax\ost)_+$ for some real $\delta>0$. Therefore,
by the definition of the positive cone of $\oss\omax\ost$ \cite[\S5]{kavruk--paulsen--todorov--tomforde2011}, there exist $n,m\in\mathbb N$,
$P\in\M_n(\oss)_+$, $Q\in\M_m(\ost)_+$, and a linear map $\alpha:\mathbb C\rightarrow\mathbb C^{nm}$ such that
\[
h\,=\,(h-\delta1) + \delta1\,=\,\alpha^*(P\otimes Q)\alpha\,=\,\sum_{i,j=1}^n\sum_{k,\ell=1}^m \overline{\alpha}_{ik}p_{ij}\otimes q_{k\ell}\alpha_{j\ell}\,,
\]
where $[\alpha_{11},\dots,\alpha_{1m},\alpha_{21},\dots,\alpha_{2m},\dots] = \alpha$.
For each pair $i,j$, set $s_{ij}=p_{ij}$ and $t_{ij}=\displaystyle\sum_{k,\ell=1}^m\overline{\alpha}_{ik}q_{k\ell}\alpha_{j\ell}$.
Then, letting $A = [\alpha_{ji}]_{i,j}$, we have $[t_{ij}]_{i,j} = A^*TA$; thus, $[t_{ij}]_{i,j} \in \cl M_n(\cl T)_+$ and
formula \eqref{expression} holds.
\end{proof}


\section{The General Framework}\label{s_osdg}

In this section, we describe the general framework for the study of operator systems of discrete groups. We recall that
if $\lgG$ is a discrete group, then $\lgG$ embeds canonically into its full group C$^*$-algebra $\cstar(\lgG)$.
The C$^*$-algebra $\cstar(\lgG)$ has the following universal property: for every unitary representation $\pi$ of $\lgG$ on a
Hilbert space $\cl H$, there exists a $*$-representation $\tilde{\pi} : \cstar(\lgG)\to\cl B(\cl H)$ such that $\tilde{\pi}(g) = \pi(g)$
for every $g\in \lgG$.

\begin{definition} Let $\mathfrak u$ be a set of generators of a discrete group $\lgG$. The \emph{operator system generated by $\mathfrak u$} is
the operator subsystem $\oss(\mathfrak u)$ of the group C$^*$-algebra $\cstar(\lgG)$ defined by
\[
\oss(\mathfrak u)\,=\,\mbox{\rm span}\{1,u,u^*\,:\,u\in\mathfrak u\}\,.
\]
\end{definition}

We recall two canonical C$^*$-covers of an operator system $\cl S$. The \emph{universal C$^*$-algebra}
$\cstaru(\cl S)$ of $\cl S$ \cite{kirchberg--wassermann1998}
can be defined by the universal property that whenever $\phi : \cl S\to\cl B(\cl H)$ is a
unital completely positive map for some Hilbert space $\cl H$, there exists a (unique) $*$-representation of
$\cstaru(\cl S)$ on $\cl H$ extending $\phi$. The \emph{enveloping C$^*$-algebra} $\cstare(\cl S)$ of $\cl S$
\cite{hamana1979b,Paulsen-book} is
defined by the universal property that whenever $\cl A$ is a C$^*$-cover of $\cl S$, there exists a $*$-homomorphism
$\pi : \cl A\to \cstare(\cl S)$ such that $\pi(x) = x$ for every $x\in \cl S$.

Let $\iota_{\rm e}$ and $\iota_{\rm u}$
denote, respectively, the unital completely isometric embeddings of $\oss(\mathfrak u)$ into $\cstare(\oss(\mathfrak u))$ and $\cstaru(\oss(\mathfrak u))$, and let
\[
\osr(\mathfrak u)\,=\,\iota_{\rm u}\left( \oss(\mathfrak u)\right)\,.
\]
Thus, $\oss(\mathfrak u)$ and $\osr(\mathfrak u)$ are completely order isomorphic operator systems
and $\cstaru(\oss(\mathfrak u))$ is generated as a C$^*$-algebra by $\osr(\mathfrak u)$.
Hence, there are three, possibly distinct, C$^*$-algebras naturally affiliated with the operator system $\oss(\mathfrak u)$:
\begin{enumerate}
\item the C$^*$-envelope $\cstare(\oss(\mathfrak u))$ of $\oss(\mathfrak u)$;
\item the C$^*$-algebra $\cstar(\lgG)$ generated by $\oss(\mathfrak u)$, and
\item the universal C$^*$-algebra $\cstaru(\oss(\mathfrak u))$.
\end{enumerate}
We note that there are a number of differences between these C$^*$-algebras.
For example, each $u\in\oss(\mathfrak u)$ is a unitary element of $\cstar(\lgG)$, but if $\tilde u=\iota_{\rm u}(u)$, then
the construction of $\cstaru(\oss(\mathfrak u))$ \cite[Proposition 8]{kirchberg--wassermann1998} shows that $\tilde u$ is never normal (unless $\lgG$ is the trivial group).
In particular, $\mathfrak u$ is not a set of unitaries in $\cstaru(\oss(\mathfrak u))$.
(As a concrete example, let $\lgG=\mathbb F_1$, the free group on a single generator $u$. Then
$\cstar(\mathbb F_1)$ is $*$-isomorphic via Fourier transform to the abelian 
C$^*$-algebra $C(\mathbb T)$ of all complex valued continuous functions on the unit circle $\mathbb{T}$. 
On the other hand, $\cstaru(\oss(\mathfrak u))$ is nonabelian,
but is singly generated; hence, the generator $\tilde u=\iota_{\rm u}(u)$ is nonnormal.)

The situation with C$^*$-envelopes is more tractable. Indeed, the following proposition is a special case of
\cite[Proposition 5.6]{kavruk2011}.

\begin{proposition}\label{cstare of oss}
Up to a $*$-isomorphism that fixes the elements of $\oss(\mathfrak u)$, we have that $\cstare(\oss(\mathfrak u))=\cstar(\lgG)$.
\end{proposition}


Since the C$^*$-envelope of $\oss(\mathfrak u)$ recaptures the group C$^*$-algebra, it is natural to ask
for a description of the C$^*$-algebras arising from the tensor products of group operator systems. 
For arbitrary operator systems we have  \cite[Theorem 6.4]{kavruk--paulsen--todorov--tomforde2011}
\[ \oss \oc \ost \coisubset \cstaru(\oss) \omax\cstaru(\ost).\]
This leads us to the following definition:

\begin{definition} Given operator systems $\oss$ and $\ost$, 
we let $\oss \oess \ost$ be the operator system 
defined by the inclusion
\[ \oss \oess \ost \coisubset \cstare(\oss) \omax \cstare(\ost).\]
\end{definition}

Since the images of $\oss$ and $\ost$ inside $\cstare(\oss) \omax
\cstare(\ost)$ under the canonical identifications
commute, we have that $\oss \oc \ost \subseteq^+ \oss \oess
\ost,$ i.e., ${\rm ess}\le {\rm c}.$  In the next section, we will prove that for free groups these two tensor products are identical. 
For the moment, we turn our attention to some conditions that guarantee the equality of ${\rm ess}$ with other tensor products.

\begin{lemma}\label{multiplication}
Let $\lgG$ and $\lgH$ be discrete groups and $\mathfrak u\subseteq \lgG$ and $\mathfrak v\subseteq \lgH$ be
finite generating sets.
Let $\tau$ be an operator system tensor product and
fix $u\in\mathfrak u$ and $v\in\mathfrak v$.
If $u\otimes 1$ or $1\otimes v$ is a unitary element of
$\cstare(\oss(\mathfrak u)\otimes_\tau \oss(\mathfrak v))$, then
$(u\otimes 1)(1\otimes v)=(1\otimes v)(u\otimes 1)=u\otimes v$.
\end{lemma}

\begin{proof}
The argument is motivated by the ideas of \cite{walter2003}. Each of the matrices
\[
p\,=\,\left[\begin{array}{ccc}1&u&u \\ u^*&1&1 \\ u^*&1&1 \end{array}\right]
\quad\mbox{and}\quad
q\,=\, \left[\begin{array}{ccc}1&1&v\\ 1&1&v \\ v^*&v^*&1 \end{array}\right]
\]
is positive in $\M_3\left(\oss(\mathfrak u)\right)$ and $\M_3\left(\oss(\mathfrak v)\right)$, respectively. Therefore, $p\otimes q$
is positive in $\M_9(\oss(\mathfrak u)\omax\oss(\mathfrak v)).$  Since $\tau \le \max,$ this element is also positive in $\M_9(\oss(\mathfrak u) \otimes_\tau \oss(\mathfrak v) ).$

Let $A\in \M_3\left(\cstare(\oss(\mathfrak u)\otimes_\tau\oss(\mathfrak v))\right)$ be given by
$$A\,=\,
\left[\begin{array}{c|cc}1\otimes 1&  u\otimes 1&u\otimes v \\ \cline{1-3}  u^*\otimes 1&1\otimes 1&1\otimes v \\ u^*\otimes v^*&1\otimes v^*&1\otimes1 \end{array}\right]
\,=\,
\left[\begin{array}{cc} 1\otimes 1 & X \\ X^* & C\end{array}\right]\,.$$
The matrix $A$ is also expressed in block form by
$$A\,=\,
\left[\begin{array}{cc|c}1\otimes 1&  u\otimes 1&u\otimes v \\  u^*\otimes 1&1\otimes 1&1\otimes v \\ \cline{1-3} u^*\otimes v^*&1\otimes v^*&1\otimes1 \end{array}\right]
\,=\,
\left[\begin{array}{cc} D & Z \\ Z^* & 1\otimes1\end{array}\right]\,.$$
Because $A=\alpha(p\otimes q)\alpha^*$ for a suitable rectangular zero-one matrix $\alpha$, the matrix $A$ is a positive element of
$\M_3(\oss(\mathfrak u)\otimes_\tau\oss(\mathfrak v))$.
By the Cholesky Algorithm,
the matrix $\left[\begin{smallmatrix} 1 & X \\ X^* & C \end{smallmatrix}\right]$ is positive if and only if $C-X^*X$ is positive and
the matrix $\left[\begin{smallmatrix} D & Z \\ Z^* & 1 \end{smallmatrix}\right]$ is positive if and only if $D-ZZ^*$ is positive. Hence,
\begin{equation}\label{e1}
0\,\leq\,C-X^*X\,=\,\left[\begin{matrix}  1-(u\otimes1)^*(u\otimes 1) &  1\otimes v - (u\otimes 1)^*(u\otimes v) \\ *   &*\end{matrix}\right]
\end{equation}
and
\begin{equation}\label{e2}
0\,\leq\,D-ZZ^*\,=\,\left[\begin{matrix}  *\hskip 6 pt &  u\otimes 1-(u\otimes v)(1\otimes v)^*\\ *\hskip 6 pt   & 1\otimes1-(1\otimes v)(1\otimes v)^*\end{matrix}\right]\,.
\end{equation}
Therefore, if $u\otimes 1$ is unitary, then inequality \eqref{e1} holds only if $1\otimes v = (u\otimes 1)^*(u\otimes v)$, and so $(u\otimes 1)(1\otimes v)=u\otimes v$.
Alternatively, if $1\otimes v$ is unitary, then inequality \eqref{e2} holds only if $u\otimes 1=(u\otimes v)(1\otimes v)^*$, which implies $(u\otimes 1)(1\otimes v)=u\otimes v$.
Hence, if $u\otimes 1$ or $1\otimes v$ is unitary, then
$(u\otimes 1)(1\otimes v)=u\otimes v$.

Repeat this argument above using
the matrices
\[
\tilde p\,=\,\left[\begin{array}{ccc}1&1&u \\ 1&1&u \\ u^*&u^*&1 \end{array}\right]
\quad\mbox{and}\quad
\tilde q\,=\, \left[\begin{array}{ccc}1&v&v\\ v^*&1&1 \\ v^*&1&1 \end{array}\right]
\]
in place of $p$ and $q$ to obtain $(1\otimes v)(u\otimes 1)=u\otimes v$.
\end{proof}

The theorem below characterises the situation in which the ambient C$^*$-algebra is not the
maximal tensor products of universal C$^*$-algebras, but rather
the maximal tensor product $\cstare\left(\oss(\mathfrak u)\right)\omax\cstare\left(\oss(\mathfrak v)\right)$ of enveloping C$^*$-algebras.

\begin{theorem}\label{l_uv}
Let $\lgG$ and $\lgH$ be discrete groups and $\mathfrak u\subseteq \lgG$ and $\mathfrak v\subseteq \lgH$ be
generating sets.
Suppose that $\tau$ is an operator system tensor product such that
$\cl S(\mathfrak u)\otimes_\tau \cl S(\mathfrak v) \subseteq^+ \cl S(\mathfrak u) \oess \cl S(\mathfrak v)$. 
Then the following statements are equivalent:
\begin{enumerate}

\item\label{ce-1} $u\otimes 1$ and $1\otimes v$ are unitary elements of
$\cstare\left(\cl S(\mathfrak u)\otimes_\tau \cl S(\mathfrak v)\right)$ for every $u\in \mathfrak u$ and $v\in \mathfrak v$;
\item\label{ce-2} up to a $*$-isomorphism,  $\cstare\left(\cl S(\mathfrak u)\otimes_\tau \cl S(\mathfrak v)\right) =\cstar(\lgG)\omax\cstar(\lgH)$;
\item\label{ce-3} $\cl S(\mathfrak u)\oess \cl S(\mathfrak v) = \cl S(\mathfrak u)\otimes_\tau \cl S(\mathfrak v)$.
\end{enumerate}
\end{theorem}

\begin{proof}
$(1) \Rightarrow(2)$. By Proposition \ref{cstare of oss},
$\cl S(\mathfrak u)\oess\cl S(\mathfrak v)\coisubset {\rm C}^*(\lgG)\omax {\rm C}^*(\lgH)$.
Suppose that
$\cstar(\lgG)\omax\cstar(\lgH)$ is represented faithfully
as a unital C$^*$-subalgebra of $\cl B(\cl H)$ for some Hilbert space $\cl H$. Using the identity maps $\iota_{\mathfrak u}$ and $\iota_{\mathfrak v}$ on 
$\cl S(\mathfrak u)$ and $\cl S(\mathfrak v)$ respectively,
we define a unital completely positive map $\vartheta: \cl S(\mathfrak u)\otimes_\tau \cl S(\mathfrak v) \to \B(\H)$ via
\[
\vartheta =\iota_{\mathfrak u}\otimes\iota_{\mathfrak v}:\cl S(\mathfrak u)\otimes_\tau \cl S(\mathfrak v)\rightarrow
\cl S(\mathfrak u)\oess \cl S(\mathfrak v)\coisubset \cstar(\lgG)\omax\cstar(\lgH).
\]
Let $\Theta:\cstare(\cl S(\mathfrak u)\otimes_\tau \cl S(\mathfrak v))\rightarrow\B(\H)$ be a completely positive extension of $\vartheta$.
The map $\Theta$ sends every contraction in $\cstare\left(\cl S(\mathfrak u)\otimes_\tau \cl S(\mathfrak v)\right)$ of the form $u\otimes 1$,
$u\in\mathfrak u$, to the unitary element $u\otimes 1$ of ${\rm C}^*(\lgG)\omax {\rm C}^*(\lgH)$.
Since the multiplicative domain of a unital completely positive map \cite{Paulsen-book}
contains all contractions that are sent to unitary elements,
$u\otimes 1$
in $\cstare\left(\cl S(\mathfrak u)\otimes_\tau \cl S(\mathfrak v)\right)$
is in the multiplicative domain of $\Theta$ for every $u\in \mathfrak u$.
Similarly, $1 \otimes v$ is in the multiplicative domain of $\Theta$ for every $v\in \mathfrak v$.
Because the multiplicative domain is an algebra, hypothesis \eqref{ce-1}  and Lemma \ref{multiplication} imply that $u\otimes v=(u\otimes1)(1\otimes v)$ is also in
the multiplicative domain of $\Theta$, for every  $u\in \mathfrak{u}$ and $v\in \mathfrak{v}$.
Because
$\cl S(\mathfrak u)\otimes \cl S(\mathfrak v)$ generates $\cstare(\cl S(\mathfrak u)\otimes_\tau \cl S(\mathfrak v))$,
the map $\Theta$ is a $*$-homomorphism of $\cstare(\cl S(\mathfrak u)\otimes_\tau\cl S(\mathfrak v))$ onto $\cstar(\lgG)\omax\cstar(\lgH)$.

Conversely, represent $\cstare(\cl S(\mathfrak u) \otimes_\tau \cl S(\mathfrak v))$ faithfully as a C$^*$-subalgebra of
$\cl B(\cl H)$ for some Hilbert space $\cl H$.
Let $\cl A \subseteq \cl B(\cl H)$ denote the C$^*$-subalgebra generated by the elements of the form $u \otimes 1, u \in \mathfrak u$, 
and let $\cl B$ be the C$^*$-subalgebra generated by elements of the form $1 \otimes v$, $v \in \mathfrak v.$ Hypothesis \eqref{ce-1} and Lemma~\ref{multiplication}
imply that $ab=ba$ whenever $a\in \cl A$ and $b\in \cl B$.

Consider the map $\rho_1 : \cl S(\mathfrak{u}) \to \cl B(\cl H)$
given by
$\rho_1(u) = u\otimes 1$, $u\in \cl S(\mathfrak{u})$, and extend it to a completely positive map on
$\cstar(\lgG),$ still denoted by $\rho_1.$ Because $\rho_1(u)$ is unitary for every $u\in \mathfrak u$,
we have that $\mathfrak u$ is contained in the multiplicative domain of $\rho_1.$
Hence, $\rho_1$ is a $*$-homomorphism and consequently its range is in $\cl A.$ Similarly,  the map $\rho_2 : \cl S(\mathfrak{v}) \to \cl B(\cl H)$
given by
$\rho_2(v) = 1\otimes v$, $v\in \mathfrak{v}$, extends to a $*$-homomorphism of $\cstar(\lgH)$ into $\cl B.$
Because $\cl A$ and $\cl B$ commute, we obtain a $*$-homomorphism $\rho_1 \otimes \rho_2 : \cstar(\lgG)\omax \cstar(\lgG) \to \cl B(\cl H),$
whose range is $\cstare (\cl S(\mathfrak u) \otimes_\tau \cl S(\mathfrak v)).$
This map is the inverse of $\Theta$ on the set of all generators and is hence $\rho_1 \otimes \rho_2 = \Theta^{-1}.$
Thus, $\Theta$ is a $*$-isomorphism between $\cstar(\lgG) \omax \cstar(\lgH)$ and $\cstare(\cl S(\mathfrak u) \otimes_\tau \cl S(\mathfrak v))$
and so (1) implies (2).

$(2)\Rightarrow(3)$.
Assumption (2) and the definition of the tensor product ${\rm ess}$ yield
\[
\cl S(\mathfrak{u})\oess\cl S(\mathfrak{v})\coisubset\cstar(\lgG)\omax\cstar(\lgH)\,=\,\cstare\left(\cl S(\mathfrak{u})\otimes_\tau\cl S(\mathfrak{v})\right)\,.
\]
On the other hand, we trivially have
$$\cl S(\mathfrak{u})\otimes_{\tau}\cl S(\mathfrak{v})\coisubset \cstare\left(\cl S(\mathfrak{u})\otimes_\tau\cl S(\mathfrak{v})\right).$$
It follows that $\cl S(\mathfrak{u})\oess\cl S(\mathfrak{v}) = \cl S(\mathfrak{u})\otimes_\tau \cl S(\mathfrak{v})$.

$(3)\Rightarrow(1)$.
Let $\iota_{\mathfrak{u}} : \cl S(\mathfrak{u})\to \cstar(\lgG)$ and $\iota_{\mathfrak{v}} : \cl S(\mathfrak{v})\to \cstar(\lgH)$
be the inclusion maps, $\phi=\iota_{\mathfrak{u}}\otimes\iota_{\mathfrak{v}}$ and
let $\osr = \phi(\cl S(\mathfrak{u})\otimes\cl S(\mathfrak{v}))\subseteq\cstar(\lgG)\otimes\cstar(\lgH)$.
By the definition of ${\rm ess}$ and Proposition \ref{cstare of oss}, $\phi$ is a complete order isomorphism of the operator system $\cl S(\mathfrak{u})\oess\cl S(\mathfrak{v})$ and
the operator subsystem $\osr\subseteq\cstar(\lgG)\omax\cstar(\lgH)$. Because $\osr$ generates $\cstar(\lgG)\omax\cstar(\lgH)$
as a C$^*$-algebra, the universal property of the C$^*$-envelope
implies that there exists a $*$-epimorphism $\pi: \cstar(\lgG)\omax\cstar(\lgH)\rightarrow \cstare(\cl S(\mathfrak{u})\oess\cl S(\mathfrak{v}))$
such that $\pi(x)=\iota_{\rm e}(x)$ for every $x\in\cl S(\mathfrak{u})\oess\cl S(\mathfrak{v})$. Because $u\otimes 1,1\otimes v\in \osr$ are unitary elements in
$\cstar(\lgG)\omax\cstar(\lgH)$, $\iota_{\rm e}(u\otimes 1)=\pi(u\otimes 1)$ and $\iota_{\rm e}(1\otimes v)=\pi(1\otimes v)$ are unitary elements of
$\cstare(\cl S(\mathfrak{u})\oess\cl S(\mathfrak{v}))$. By invoking the hypothesis
$\cl S(\mathfrak{u})\oess\cl S(\mathfrak{v}) = \cl S(\mathfrak{u})\otimes_\tau\cl S(\mathfrak{v})$,  we have that
$u\otimes 1$ and $1\otimes v$ are unitary in $\cstare\left(\cl S(\mathfrak{u})\otimes_\tau \cl S(\mathfrak{v})\right)$, for every $u\in \mathfrak{u}$ and
every $v\in \mathfrak{v}$.
\end{proof}

\begin{corollary}\label{oess=oc}
Let
$\mathfrak u$ and $\mathfrak v$ be generating sets of discrete groups $\lgG$ and $\lgH,$ respectively.
Then $\cl S(\mathfrak u) \oess \cl S(\mathfrak v) = \cl S(\mathfrak u) \oc \cl S(\mathfrak v)$ if and only if
 $u\otimes 1$ and $1\otimes v$ are unitary elements of
$\cstare\left(\cl S(\mathfrak u)\oc \cl S(\mathfrak v)\right)$ for every $u\in \mathfrak u$ and $v\in \mathfrak v.$
\end{corollary}

Several discrete groups satisfy the hypothesis
 of Theorem \ref{l_uv}, as we shall see in subsequent sections of the paper.


\section{Operator Systems on Finitely Generated Free Groups}\label{s_osfgfg}

In this section, we study in detail the operator systems of free groups.
Note that definition of $\oss(\mathfrak u)$ depends in general on the choice of the
generating set $\mathfrak u\subseteq\lgG$. We first show that with free groups, we can dispense with this dependence.

\begin{proposition}\label{ind defn} If $\mathfrak u=\{u_1,\dots,u_n\}$ and $\mathfrak v=\{v_1,\dots, v_n\}$ are two sets of generators of the free group $\fn$, then
there is a complete order isomorphism $\phi:\oss(\mathfrak u)\rightarrow\oss(\mathfrak v)$ such that $\phi(u_j)=v_j$,
$j=1,\dots,n$.
\end{proposition}

\begin{proof} By the universality of $\cstar(\fn)$, there is a $*$-automorphism $\vartheta$ of $\cstar(\fn)$ for which
$\vartheta(u_j)=v_j$, $j=1,\dots,n$. Thus, $\phi=\vartheta_{\vert\oss(\mathfrak u)}$ is a complete order isomorphism from $\oss(\mathfrak u)$
onto $\oss(\mathfrak v)$.
\end{proof}

Henceforth, let $\mathfrak u=\{u_1,\dots,u_n\}$ and $\mathfrak v=\{v_1,\dots,v_m\}$ be generators for the free groups $\fn$ and $\fm$ respectively
and denote by $\oss_n$ and $\oss_m$, respectively,  the operator subsystems
$\oss(\mathfrak u)\subseteq\cstar(\fn)$ and
$\mathfrak \oss(\mathfrak v)\subseteq\cstar(\fm)$. By Proposition \ref{ind defn}, we may identify $\oss_n$ and $\oss_m$
in the case where $m=n$. Moreover, we write
\[
\oss_n\,=\,\mbox{\rm Span}\,\{u_{-n},\dots,u_{-1},u_0, u_1,\dots, u_n\}\,,
\]
where $u_{-k}=u_k^*$, $k=1,\dots,n$, and $u_0=1$.

Throughout, $\iota_k$ shall denote the inclusion of $\oss_k$ into $\cstar(\fk)$.


 \begin{lemma}\label{c-max coi}
 $\oss_n\oc\oss_m\coisubset\cstar(\fn)\omax\cstar(\fm)$ and, consequently, $\oss_n\oc \oss_m = \oss_n \oess \oss_m.$
\end{lemma}

\begin{proof} Since $\osr\oc\oss_n\coisubset\osr\oc\cstar(\fn)$
for every $n\in\mathbb N$ and every operator system $\osr$ \cite[Lemma 4.1]{farenick--kavruk--paulsen2011},
and because ${\rm c} = \max$ if one of the tensor factors is a unital C$^*$-algebra
\cite[Theorem 6.7]{kavruk--paulsen--todorov--tomforde2011}, we deduce that
\[
\oss_n\oc\oss_m \coisubset \oss_n\omax\cstar(\fm)
=\oss_n \oc\cstar(\fm)
\coisubset \cstar(\fn) \omax \cstar(\fm)\,.
\]
Hence, $\oss_n\oc\oss_m\coisubset\cstar(\fn)\omax\cstar(\fm)$.
The last statement follows by the definition of ${\rm ess}$ and Proposition \ref{cstare of oss}.
\end{proof}

\begin{proposition}\label{c-min}
Let $\cl R$ be an operator system.
Then $\cl R\omin \oss_1 = \cl R \oc \oss_1$. In particular,
$\oss_1\omin\oss_1=\oss_1\oc\oss_1$.
\end{proposition}
\begin{proof}
Since $\cstar(\mathbb{F}_1)$ is $*$-isomorphic to $C(\mathbb{T})$, we have that
$\cl R\otimes_{\min}\cl S_1 \coisubset \cl R\otimes_{\min} C(\mathbb{T})$.
On the other hand, \cite[Lemma 4.1]{farenick--kavruk--paulsen2011} implies that
$\cl R\oc\cl S_1\coisubset \cl R\oc C(\mathbb{T})$.
Since $C(\mathbb{T})$ is a nuclear C$^*$-algebra, and hence a nuclear operator system, we conclude that
$\cl R\otimes_{\min} C(\mathbb{T}) = \cl R\oc C(\mathbb{T})$, and thus
$\cl R\omin \oss_1 = \cl R \oc \oss_1$.
\end{proof}

In the lemma below, the vector spaces $\M_p(\osr\otimes\ost)$ and $\M_p(\osr)\otimes\ost$ are canonically identified,
and $\phi_m$ denotes the map defined in Example \ref{quo map}.

\begin{lemma}\label{strictly pos lift}
If $Y=\displaystyle\sum_{\ell=-m}^m Y_\ell\otimes u_\ell$ is strictly positive in $\M_p(\oss_n\omax\oss_m)$, then there
exist $A_{ij}\in\M_p(\cstar(\fn))$, where $i,j\in\{1,\dots,m+1\}$ and $|i-j|\leq 1$, such that:
\begin{enumerate}
\item[{(i)}] each $A_{ii}$ is strictly positive for each $i$ and $\frac{1}{m}\sum_{i=1}^{m+1} A_{ii}=Y_0$;
\item[{(ii)}] $A_{j,j+1}=Y_j$ and $A_{j+1,j}=Y_{-j}$, for all $j=1,\dots,m$;
\item[{(iii)}] $X=\displaystyle\sum_{|i-j|\leq1} A_{ij}\otimes E_{ij}$ is strictly positive in $\M_p(\cstar(\fn))\omax\ost_{m+1}$;
\item[{(iv)}] $Y=\left({\rm id}_{\cstar(\fn)}\otimes\phi_m\right)^{(p)}(X)$.
\end{enumerate}
\end{lemma}

\begin{proof} Recall that  ${\rm c}=\max$ if
one of the tensor factors is a C$^*$-algebra  \cite[Theorem 6.7]{kavruk--paulsen--todorov--tomforde2011}. Thus,
\[
\oss_n\omax\oss_m\,\subseteq^+\,\oss_n\oc\oss_m\coisubset\cstar(\fn)\oc\oss_m \,=\,\cstar(\fn)\omax\oss_m\,,
\]
and so the strictly positive element $Y\in \M_p(\oss_n\omax\oss_m)$ is also strictly positive in $\M_p(\cstar(\fn)\omax\oss_m)$.
Let $\vartheta_{n,m}:\cstar(\fn)\omax\ost_{m+1}\rightarrow\cstar(\fn)\omax\oss_m$ be given by
$\vartheta_{n,m}={\rm id}_{\cstar(\fn)}\otimes\phi_m$. By \cite[Proposition 1.6, Theorem 4.2]{farenick--paulsen2011},
$\vartheta_{n,m}$ is a complete quotient map of $\cstar(\fn)\omax\ost_{m+1}$ onto $\cstar(\fn)\omax\oss_m$.
By Proposition \ref{quo map criterion},
there is a strictly positive element
$X\in\M_p(\cstar(\fn)\omax\ost_{m+1})$ such that $\vartheta_{n,m}^{(p)}(X)=Y$.
Express $X$ as
\[
X\,=\,\sum_{|i-j|\leq 1} A_{ij}\otimes E_{ij}\,,
\]
for some $A_{ij}\in\M_p(\cstar(\fn))$. Thus,
\[
\begin{array}{rcl}
Y\,=\, \vartheta_{n,m}^{(p)}(X)&=&\displaystyle\sum_{|i-j|\leq 1}A_{ij}\otimes \phi_m(E_{ij})  \\ && \\
&=&\displaystyle\frac{1}{m}\sum_{i=1}^{m+1}A_{ii}\otimes u_0\,+\,
\frac{1}{m}\displaystyle\sum_{j=1}\left[(A_{j+1,j}\otimes u_{-j})+(A_{j,j+1}\otimes u_{j})\right]\,,
\end{array}
\]
which implies that $\frac{1}{m}\sum_{i=1}^{m+1} A_{ii}=Y_0$,
$A_{j,j+1}=Y_j$ and $A_{j+1,j}=Y_{-j}$, for all $j=1,\dots,m$.
Since each $A_{ii}$ is attained from $X$
by a formal matrix product $QXQ$ for a projection $Q\in\M_p(\mathbb C)$, each $A_{ii}$ is
is strictly positive.
\end{proof}

The image of $\cl S_1$ under the identification of $\cstar(\mathbb{F}_1)$ with $C(\mathbb{T})$
is ${\rm span}\{1,z,\overline{z}\}$, where $z$ is the identity function.
By abuse of notation, we will use the symbol $z$ to denote the variable in $\mathbb{T}$,
which will cause no confusion as it will be clear from the context whether we refer to the identity function or the
corresponding variable.
If $\cl T$ is any vector space then $\cl T\otimes\cl S_1$ will be identified in a natural way with the space
of all functions from $\mathbb{T}$ into $\cl T$ of the form $\alpha_0 + \alpha_1 z + \alpha_2 \overline{z}$, $\alpha_0,\alpha_1,\alpha_2\in \cl T$.

\begin{corollary}\label{spl for} For every strictly positive element $Y\in \cl M_2(\oss_1\omax\oss_1)$, where
\[
Y\,=\,Y_0\otimes 1 + Y_1\otimes z + Y_1^*\otimes \overline{z}
\]
for some $Y_0,Y_1\in\M_2(\oss_1)$, there exist strictly positive elements $A,B\in\M_2(C(\mathbb T))$ such that
$A+B=Y_0$ and
\[
X=\left[\begin{array}{cc} A& Y_1 \\ Y_1^* & B\end{array}\right]
\]
is strictly positive in
$\M_4(C(\mathbb T))$.
\end{corollary}

\begin{proof}
By Lemma \ref{strictly pos lift}, there exists such an $X$ in $\cl M_2(\cstar(\mathbb F_1)\omax\ost_2)=\M_2(C(\mathbb T)\omax\cl T_2)$.
But, by the nuclearity of $C(\mathbb T)$, we have that
\[
\M_2(C(\mathbb T)\omax\cl T_2)\,=\,\M_2(C(\mathbb T)\omin\cl T_2)\,
\coisubset \M_2(C(\mathbb T)\omin\cl M_2)\,=\,\M_4(C(\mathbb T))\,,
\]
which yields the result.
\end{proof}

In Corollary \ref{spl for} above, the diagonal elements of the matrix $\left[\begin{smallmatrix} A& Y_1 \\ Y_1^* & B \end{smallmatrix}\right]$
belong to $\M_2\left(C(\mathbb T)\right)$. In fact, it is possible
for these diagonal entries to be chosen from $\M_2(\oss_1)$, as we now demonstrate.

\begin{lemma}\label{step 2} Let $\ost$ be an operator system. If, for some $t_0,t_1\in\ost$, the element
\begin{equation}\label{element}
y\,=\,1\otimes t_0 + z\otimes t_1 + \overline{z}\otimes t_1^*  \in  \oss_1\omax\ost \,,
\end{equation}
is strictly positive, then there are $t_0^1,t_0^2\in\ost_+$ such that $t_0^1+t_0^2=t_0$ and
\begin{equation}\label{2x2 matrix}
Y\,=\,\left[ \begin{array}{cc} t_0^1 & t_1 \\ t_1^* & t_0^2\end{array} \right] \in \M_2(\ost)_+\,.
\end{equation}
Conversely, if $Y\in\M_2(\ost)_+$ is the matrix \eqref{2x2 matrix},
then the element $y$ in \eqref{element}, where $t_0=t_0^1+t_0^2$, is positive in $\oss_1\omax\ost$.
\end{lemma}

\begin{proof}
Let $1\otimes t_0 + z\otimes t_1 + \overline{z}\otimes t_1^*$ be strictly positive in $\oss_1\omax\ost$.
By Lemma \ref{f-dim max},
there exist $n\in\mathbb N$, $F=[f_{ij}]_{i,j}\in\M_n(\oss_1)_+$, and $T=[t_{ij}]_{i,j}\in\M_n(\ost)_+$ such that
\[
1\otimes t_0  + z\otimes t_1 + \overline{z}\otimes t_1^*\,=\,\sum_{i,j=1}^n f_{ij}\otimes t_{ij}\,.
\]
Write $f_{ij}=\alpha_{ij} + \beta_{ij}z+\overline{\beta_{ij}}\overline{z}$, for some $\alpha_{ij},\beta_{ij}\in \mathbb{C}$.
Thus,
\[
1\otimes t_0 + z\otimes t_1 + \overline{z}\otimes t_1^*
\,=\,\sum_{i,j=1}^n (\alpha_{ij} + \beta_{ij}z+\overline{\beta_{ij}}\overline{z})\otimes t_{ij}\,,
\]
which implies that
$$1\otimes t_0 =  \displaystyle\sum_{i,j=1}^n \alpha_{ij}\otimes t_{ij}\,=\, 1 \otimes  \displaystyle\sum_{i,j=1}^n \alpha_{ij} t_{ij}$$
and
$$z\otimes t_1 =  \displaystyle\sum_{i,j=1}^n \beta_{ij}z\otimes t_{ij}\,=\, z \otimes  \displaystyle\sum_{i,j=1}^n \beta_{ij}t_{ij}.$$
Hence,
\[
t_0  \,=\,\displaystyle\sum_{i,j=1}^n \alpha_{ij} t_{ij}\quad\mbox{and}\quad
t_1\,=\, \displaystyle\sum_{i,j=1}^n \beta_{ij}t_{ij}\,.
\]

Let $A =[\alpha_{ij}]_{i,j}$ and $B=[\beta_{ij}]_{i,j}\in\M_n(\mathbb C)$.
Since $F(z)\in\M_n(\CC)_+$ for every $z\in\mathbb T$, the matrix $A$ is positive while
$\overline{B}=B^*$.
For every $\delta > 0$, the matrix
$A_{\delta} := A + \delta I$ is positive and invertible and hence
\[
A_{\delta}^{-1/2}(F(z) + \delta I)A_{\delta}^{-1/2}\,=\,I+A_{\delta}^{-1/2}BA_{\delta}^{-1/2}z + A_{\delta}^{-1/2}B^*A_{\delta}^{-1/2}\overline{z}
\]
is positive for every $z\in\mathbb T$. It follows that $\frac{1}{2}$ is an upper bound on the numerical radius
of $A_{\delta}^{-1/2}BA_{\delta}^{-1/2}$ and hence, by Ando's theorem \cite{ando1973},
there exist matrices $Q_{1,\delta},Q_{2,\delta}\in\M_n(\CC)_+$ such that $Q_{1,\delta}+Q_{2,\delta}=I$ and
$
\left[ \begin{smallmatrix} Q_{1,\delta} & A_{\delta}^{-1/2}BA_{\delta}^{-1/2} \\ A_{\delta}^{-1/2}B^*A_{\delta}^{-1/2} & Q_{2,\delta}\end{smallmatrix}\right]
$
is positive. Therefore, if $A_{j,\delta}=A_{\delta}^{1/2}Q_{j,\delta}A_{\delta}^{1/2}$, $j = 1,2$, then $A_{1,\delta}+A_{2,\delta} = A_{\delta}$ and
$\left[ \begin{smallmatrix} A_{1,\delta} & B \\ B^* & A_{2,\delta}\end{smallmatrix}\right]$
is positive.
Since the matrices $A_{1,\delta}$, $A_{2,\delta}$ are uniformly bounded for $\delta \in (0,1)$, by passing to a limit point we obtain positive matrices
$A_1$ and $A_2$ such that $A_1 + A_2 = A$ and
$\left[ \begin{smallmatrix} A_{1} & B \\ B^* & A_{2}\end{smallmatrix}\right]$
is positive.
By Lemma \ref{schur},
$\left[ \begin{smallmatrix} A_1\circ T & B\circ T \\ B^* \circ T& A_2\circ T\end{smallmatrix}\right]$ is positive in $\M_{2n}(\ost)$.

For any operator system $\osr$, the map $\sigma_n:\M_n(\osr)\rightarrow\osr$ defined by
$\sigma_n([r_{ij}]_{i,j})=\sum_{i,j} r_{ij}$ is completely positive. Hence, with $(\alpha_{ij}^{[k]})_{i,j}=A_k$ and $t_0^k=\sum_{i,j}\alpha_{ij}^{[k]}t_{ij}$, for $k=1,2$,
we have that
\[
\left[ \begin{array}{cc} t_0^1 & t_1\\ t_1^*& t_0^2\end{array}\right]
\,=\,\left[ \begin{array}{cc} \sum_{ij}\alpha_{ij}^{[1]}t_{ij} & \sum_{i,j}\beta_{ij}t_{ij} \\ \sum_{i,j}\overline{\beta}_{ij}t_{ij}^*& \sum_{i,j}\alpha_{ij}^{[2]}t_{ij}\end{array}\right]
\,=\,\sigma_n^{(2)}\left( \left[ \begin{array}{cc} A_1\circ T & B\circ T \\ B^* \circ T& A_2\circ T\end{array}\right]\right)
\]
is positive in $\M_2(\ost)$.

Conversely, assume that $Y=\left[ \begin{smallmatrix} t_0^1 & t_1 \\ t_1^* & t_0^2\end{smallmatrix} \right] \in \M_2(\ost)_+$, and set $t_0=t_0^1+t_0^2$.
Because matrix $Z=\left[ \begin{smallmatrix} 1 & z \\ \overline{z} & 1\end{smallmatrix} \right]$ is positive in $\M_2(\oss_1)$, the element
$y=1\otimes (t_0^1+t_0^2) + z\otimes t_1 + \overline{z}\otimes t_1^*$ is positive in $\oss_1\omax\ost$
\cite[Lemma 5.2]{kavruk--paulsen--todorov--tomforde2011}.
\end{proof}

We now arrive at one of the main results of this paper.
Recall that $\oss_1\omin\oss_1=\oss_1\oc\oss_1$ (Proposition \ref{c-min}).
The following theorem, therefore, shows that $\oss_1\oc\oss_1 \neq \oss_1\omax\oss_1$.

\begin{theorem}\label{sone} If $\iota={\rm id}_{\oss_1}\otimes{\rm id}_{\oss_1}$,
then $\iota:\oss_1\omin\oss_1\rightarrow\oss_1\omax\oss_1$ is a positive linear map that is not $2$-positive.
\end{theorem}

\begin{proof}
To prove that $\iota$ is a positive map, note that $\oss_1 \omin \oss_1 \coisubset C(\mathbb T)\omin C(\mathbb T)=C(\mathbb T \times \mathbb T)$.
Let $z$ and $w$ denote the images of the two generators, viewed as functions on the torus $\mathbb T^2$.
We may write a typical element of  $\oss_1 \omin \oss_1$ as $f = b_{-1} \otimes \overline{z} + b_0\otimes 1 + b_1\otimes z$,
where $b_i \in \mbox{span} \{ \overline{w}, 1, w \}$.
Such an element $f$ is positive as a function on the torus if and only if $b_{-1} = \overline{b_1}$ and $b_0 \ge 2 |b_1|$.
But these conditions guarantee that the matrix
\[ \left[\begin{array}{cc} b_0/2 & b_1\\ b_{-1} & b_0/2 \end{array}\right],\]
is in $\cl M_2(\oss_1)_+$.
We also have that
\[ \left[\begin{array}{cc} 1 & z\\ \overline{z} & 1 \end{array}\right] \]
belongs to $\cl M_2(\oss_1)_+$. Thus,
\[ f = b_0/2 \otimes 1  + b_1\otimes z + b_{-1}\otimes \overline{z} + b_0/2\otimes 1 \]
is in $(\oss_1 \omax \oss_1)_+$ \cite[Lemma 5.2]{kavruk--paulsen--todorov--tomforde2011}. That is,
every element of $(\oss_1 \omin \oss_1)_+$ is also in $(\oss_1 \omax \oss_1)_+$,
which proves the positivity of the map $\iota$.

To show that $\iota$ is not $2$-positive, consider the element $h\in\M_2(\oss_1\omin\oss_1)$ which, considered as an $\cl M_2(\mathbb{C})$-valued
function on the variables $z,w\in \mathbb{T}$, is given by
\begin{equation}\label{not in}
h(z,w)\,=\,
\left[\begin{array}{cc} 3+2\Re(zw) & 2\overline{z}w \\ 2z\overline{w} & 3-2\Re(zw) \end{array}\right]\,.
\end{equation}
Because $\oss_1\omin\oss_1\coisubset C(\mathbb T)\omin C(\mathbb T)$, $h$ is positive in $\M_2(\oss_1\omin\oss_1)$
if and only if $h(z,w)\in\M_2(\CC)_+$ for every $z,w\in\mathbb T$.
The matrix
 \[
k(z,w)\,=\,\left[\begin{array}{cc} 2\Re(zw) & 2\overline{z}w \\ 2z\overline{w} & -2\Re(zw) \end{array}\right]\,,
\]
has characteristic polynomial $\lambda^2-\left(2\Re(zw)\right)^2-4$. Thus,
the eigenvalues of the hermitian matrix $h(z,w)=3I+k(z,w)$ are uniformly bounded below by $3-2\sqrt{2}>0$, which shows that
$h$ is strictly positive in $\M_2\left(\oss_1\omin\oss_1\right)$. We now show that $h$ is not a (strictly) positive element of $\M_2(\oss_1\omax\oss_1)$.

Suppose, contrary to what we aim to prove, that $h$ is in fact strictly positive in $\M_2(\oss_1\omax\oss_1)$.
By the identification of $\M_2(\oss_1\otimes\oss_1)$ with $\oss_1\otimes\M_2(\oss_1)$, rewrite
$h$ in tensor notation:
\begin{equation}\label{f-2}
h\,=\,1\otimes \left[ \begin{array}{cc} 3&0\\ 0&3 \end{array}\right]
+
z \otimes \left[ \begin{array}{cc} w&0\\ 2\overline{w}&-w \end{array}\right]
+
\overline{z}\otimes  \left[ \begin{array}{cc} \overline{w}&2w\\ 0&-\overline{w} \end{array}\right]\,.
\end{equation}
Thus, because $h$ is strictly positive, Lemma \ref{step 2} implies that
there is a strictly positive
\[
X=\left[\begin{array}{cc} A& C \\ C^* & B\end{array}\right]
\in\M_4(\oss_1)
\]
such that $A+B=\left[ \begin{smallmatrix} 3&0\\ 0&3 \end{smallmatrix}\right]$ and
$C(w)=\left[ \begin{smallmatrix} w&0\\ 2\overline{w}&-w \end{smallmatrix}\right]$ in $\M_2(\oss_1)$.
Let $R=U^*XU$, where $U=\mbox{Diag}(w^2,1,w,\overline{w})\in\M_4\left(C(\mathbb T)\right)$. Thus, as a function $R:\mathbb T\rightarrow\M_4(\mathbb{C})$,
$R$ is given by
\[
R(w)\,=\,
\left[
\begin{array} {cc|cc}
f_{11}(w) & \overline{w}^2f_{12}(w)  & 1 & 0 \\
w^2 f_{21}(w) & f_{22}(w) & 2 & -1 \\
\cline{1-4}
1 & 2 & 3-f_{11}(w) & -\overline{w}^2 f_{12}(w)\\
0 & -1 & -w^2 f_{21}(w) & 3-f_{22}(w)
\end{array}
\right]\,,
\]
where $f_{ij}\in\oss_1$ and is of the form $f_{ij}(w)=\alpha_{ij}+\beta_{ij}w+\gamma_{ij}\overline{w}$, for $i,j\in\{1,2\}$. Let
$g_{21}(w)=w^2 f_{21}(w)=\alpha_{21}w^2+\beta_{21}w^3+\gamma_{21}w$. Because $R(w)$ is hermitian, $\overline{w}^2f_{12}(w)=\overline{g_{21}(w)}$.

Let $\zeta=e^{i\frac{2\pi}{5}}$. Then $\zeta$, $\zeta^2$, and $\zeta^3$ all generate the same cyclic subgroup $C$ of $\mathbb T$, namely $C\cong\mathbb Z_5$.
Because $\zeta+\zeta^2+\cdots+\zeta^5=0$, we deduce that
\[
\frac{1}{5}\sum_{k=1}^5 f_{jj}(\zeta^k)\,=\,\alpha_{jj}
\quad\mbox{and}\quad
\frac{1}{5}\sum_{k=1}^5 g_{21}(\zeta^k)\,=\,0\,.
\]
Thus, in $\M_4(\CC)$,
\[
0\,\leq\,\frac{1}{5}\sum_{k=1}^5 R(\zeta^k)\,=\,\left[
\begin{array} {cc|cc}
\alpha_{11} & 0  & 1 & 0 \\
0& \alpha_{22} & 2 & -1 \\
\cline{1-4}
1 & 2 & 3-\alpha_{11}& 0\\
0 & -1 &0 & 3-\alpha_{22}
\end{array}
\right]\,.
\]
 Let $b\in\M_4(\CC)$ denote the positive matrix above, which in $2\times 2$ block form
 we write as $b=\left[\begin{smallmatrix} \alpha_{11} & b_{12} \\ b_{12}^* & b_{22} \end{smallmatrix}\right]$.
 By the Cholesky Algorithm, there is a matrix $Z$ for which $b=Z^*\left[\begin{smallmatrix} \alpha_{11} & 0 \\ 0 &\tilde b_{22} \end{smallmatrix}\right]Z$, where
 \[
0\,\leq\, \tilde b_{22}\,=\,b_{22}-\frac{1}{\alpha_{11}}b_{12}^*b_{12}\,=\,\left[\begin{array}{ccc}
\alpha_{22} & 2 & -1 \\ 2 & 3-\alpha_{11}-\frac{1}{\alpha_{11}} & 0 \\ -1 & 0 & 3-\alpha_{22}
\end{array}\right]\,.
\]
As $2\leq\alpha_{11}+\frac{1}{\alpha_{11}}$, the matrix
\[
c\,=\,\left[\begin{array}{ccc}
\alpha_{22} & 2 & -1 \\ 2 & 1 & 0 \\ -1 & 0 & 3-\alpha_{22}
\end{array}\right]
\]
is positive. Thus,
\[
0\leq\det(c)=-\alpha_{22}^2+7\alpha_{22}-13
\,.
\]
The polynomial $p(x)=x^2-7x+13$ has no real roots and $p(0)=13>0.$ Thus, $p(x)>0$ for every $x\in\mathbb R$ contradicting $p(\alpha_{22}) \leq 0$.
Therefore, $h$ is not a positive element of $\M_2(\oss_1\omax\oss_1)$.
 \end{proof}

Some important applications of Theorem \ref{sone} are given below in Theorems \ref{ocomax} and \ref{omaxe}.

\begin{theorem}\label{ocomax} 
The identity map from $\oss_n \oc \oss_m$ to $\oss_n \omax \oss_m$ is not 2-positive for all $n,m \ge 1$.
\end{theorem}
\begin{proof}
Since $\cstar(\mathbb{F}_1)$ can be canonically identified with a C$^*$-subalgebra of $\cstar(\mathbb{F}_n)$,
the canonical map $\iota_n:\oss_1 \to \oss_n$ is a complete order inclusion.
By the universal property of $\cl S_n$ \cite[Proposition 9.7]{kavruk--paulsen--todorov--tomforde2010},
we have a completely positive map $\psi_n: \oss_n \to \oss_1$ defined by sending the generators $\{u_2,...,u_n, u_{-2},..., u_{-n} \}$ to $0$
and fixing $u_1$ and $1$.
Thus, if for any $n,m$ we had that $\gamma: \oss_n \oc \oss_m \to \oss_n \omax \oss_m$ were 2-positive, then we would have that
$(\psi_n \otimes \psi_m) \circ \gamma \circ (\iota_n \otimes \iota_m): \oss_1 \oc \oss_1 \to \oss_1 \omax \oss_1$ is 2-positive, 
contradicting in this way Theorem~\ref{sone}.
\end{proof}

\begin{theorem}\label{omaxe}
We have that $\cstare(\oss_n\omax\oss_m) \neq \cstare(\oss_n)\omax\cstare(\oss_m)$ for every $n,m\geq 1$.
\end{theorem}
\begin{proof} By Proposition \ref{cstare of oss}, $\cstare(\oss_n)=\cstar(\fn)$.
By Lemma \ref{c-max coi}, $\oss_n\oc\oss_m\coisubset\cstar(\fn)\omax\cstar(\fm)$;
in other words, $\oss_n\oc\oss_m = \oss_n\otimes_{\rm ess}\oss_m$.
The hypothesis of Theorem \ref{l_uv} is clearly satisfied for $\tau = \max$.
By Theorem \ref{ocomax}, $\oss_n\oc\oss_m\neq\oss_n\omax\oss_m$, and hence Theorem \ref{l_uv}
yields $\cstare(\oss_n\omax\oss_m) \neq \cstar(\fn)\omax\cstar(\fm)$.
\end{proof}

\begin{proposition} The map $\phi_2 \omin \phi_2: \M_2 \omin \M_2 \to \oss_1 \omin
  \oss_1$ is not a complete quotient map.
\end{proposition}
\begin{proof} Since $\phi_2 \omax \phi_2: \M_2 \omax \M_2 \to \oss_1 \omax
  \oss_1$ is completely positive and $\M_2 \omin \M_2 = \M_2 \omax \M_2$,
  if $\phi_2 \omin \phi_2$ were a complete quotient map, then we could lift every positive
  element in $P \in \M_2(\oss_1 \omin \oss_1)_+$ to a positive element in
  $R \in \M_2(\M_2 \omin \M_2)_+$, which would imply that $(\phi_2 \omax \phi_2)(R)$ is
  positive in $\oss_1 \omax \oss_1$. However, this would show that the
  identity map from $\oss_1 \omin \oss_1$ to $\oss_1 \omax \oss_1$ is
  2-positive, contradicting Theorem \ref{sone}.
\end{proof}

Let $\mathbb S_n$ denote the group of permutations of $\{1,\dots,n\}$
and select $\sigma\in\mathbb S_n$. The linear map $\phi_\sigma: \oss_n
\to \oss_n$ defined on the canonical basis of $\oss_n$ by $\phi_\sigma(u_i) = u_{\sigma(i)}$,
is the restriction of a $*$-automorphism of $\cstar(\lgG)$ and is hence
a complete order isomorphism.
Fix another element $\tau\in\mathbb S_m.$
Since $\phi_{\sigma}$ and $\phi_{\tau}$ are complete order isomorphisms, for any
tensor product $\alpha$ that is {\em functorial} in the sense of
\cite{kavruk--paulsen--todorov--tomforde2011}, the canonical tensor product map
$\phi_{\sigma} \otimes_{\alpha} \phi_{\tau}: \oss_n \otimes_{\alpha}
\oss_m \to \oss_n \otimes_{\alpha} \oss_m$ is a complete
order isomorphism.  It was shown in
\cite{kavruk--paulsen--todorov--tomforde2011} that $\min,
{\rm c}$ and $\max$ are all functorial. We note that we do not know whether the
tensor product ${\rm ess}$ is functorial.

\begin{proposition}\label{no unitary}  In $\cstare(\oss_n\omax\oss_m)$,
$u_i\otimes 1$ fails to be a unitary element for every $i$ and
$1\otimes v_j$ fails to be a unitary element for every $j$.
\end{proposition}

\begin{proof} If $\sigma\in\mathbb S_n$, then $\phi_{\sigma} \otimes {\rm id}$ is a complete order isomorphism
$\oss_n\omax\oss_m\rightarrow \oss_n\omax\oss_m$, by
the preceding observation. Therefore, by the universal property of C$^*$-envelopes, there
is an epimorphism $\pi_\sigma:\cstare(\oss_n\omax\oss_m)\rightarrow
\cstare\left(\oss_n\omax\oss_m\right)$
such that $u_{\sigma(i)}\otimes 1=\pi_\sigma(u_i\otimes 1)$ for every $i$. Thus, if $u_i\otimes 1$ is a unitary element of
$\cstare(\oss_n\omax\oss_m)$ for some $i$, then $u_\ell\otimes 1$ is unitary for every $\ell$.
A similar statement holds for the elements $1\otimes v_j$.
But, by Theorems \ref{l_uv} and \ref{ocomax}, at least one element in the set $\{u_i\otimes 1,1\otimes v_j\,:\,1\leq i\leq n,\;1\leq j\leq m\}$
fails to be unitary in $\cstare(\oss_n\omax\oss_m)$. Hence, $u_i\otimes 1$ fails to be a unitary element of $\cstare(\oss_n\omax\oss_m)$ for every $i$ or
$1\otimes v_j$ fails to be a unitary element of $\cstare(\oss_n\omax\oss_m)$ for every $j$.

Now consider the case where $n=m$ and let $\sigma: \oss_n \omax \oss_n \to \oss_n \omax \oss_n$ be the ``flip'' map $\sigma(x \otimes y) = y \otimes x.$ This map is a complete order isomorphism and hence induces a $*$-automorphism
$\pi: \cstare(\oss_n \omax \oss_n) \to \cstare(\oss_n \omax \oss_n)$ with $\pi(u_i \otimes 1) = 1 \otimes v_i.$ Hence, in this case both $u_i \otimes 1$ and $1 \otimes v_j$ must be non-unitary.

Now let $n <m$ and note that, as in the proof of Theorem \ref{ocomax},
the linear map $\iota: \oss_n \to \oss_m,$ given by $\iota(u_j) = v_j$, is a unital complete order isomorphism onto its range and
the linear map $\phi: \oss_m \to \oss_n$ defined by $\phi(v_j) = u_j, 1 \le j \le n$ and $\phi(v_j) =0, n<j \le m$, is a unital completely positive map
that is a left inverse of $\iota$.
By the functoriality of $\max$, the map ${\rm id}_n \otimes \iota: \oss_n \omax \oss_n \to \oss_n \omax \oss_m$ is completely positive
with completely positive left inverse ${\rm id}_n \otimes \phi$.
Thus, the image
$\cl R= {\rm span} \{ 1 \otimes 1, u_i \otimes 1, 1 \otimes v_j, u_i \otimes v_j: 1 \le i,j \le n \}$
of ${\rm id}_n \otimes \iota$ inside $\oss_n \omax \oss_m$
is an operator system that is completely order isomorphic to $\oss_n \omax \oss_n.$

Let $\cstar(\cl R)$ be the C$^*$-subalgebra of $\cstare(\oss_n \omax \oss_m)$ generated by $\cl R.$ By the universal
property of the C$^*$-envelope, there exists a surjective $*$-homomorphism
$\pi: \cstar(\cl R) \to \cstare(\oss_n \omax \oss_n)$ which fixes
$u_i \otimes 1$ and $1 \otimes v_j$ for $1 \le i,j \le n.$
If $u_i \otimes 1$ or $1 \otimes v_j$ were unitary in $\cstare(\oss_n \omax \oss_m)$ for some $i,j = 1,\dots,n$, 
then its image would be unitary in
$\cstare(\oss_n \omax \oss_n).$  Hence, $u_i \otimes 1$ and $1 \otimes v_j$ must be non-unitary for all $i,j = 1,\dots,n$
Permuting the generators of $\cl S_m$ shows that $1\otimes v_j$ fails to be unitary in 
$\cstare(\oss_n \omax \oss_n)$ for all $j = 1,\dots,m$. 
\end{proof}


\section{The Universal Operator Systems $\oss_1$ and $\oss_1^d$ and their Tensor Products}\label{op th}

The operator systems $\oss_1$ and $\oss_1^d$ represent certain universal objects in operator theory, which we explain in Theorems \ref{universal contraction} 
and \ref{universal num-rad contraction} below. The dual system $\oss_1^d$ admits a concrete representation as a matrix operator system: namely,
 $\oss_1^d$ is (completely order isomorphic to) the 
operator subsystem $\mbox{Span}\{I,E_{12}, E_{21}\}$ of $\M_2$
 \cite[Theorem 4.4]{farenick--paulsen2011}.

 The first result in this direction is certainly well-known, but we state and prove it for completeness.

\begin{proposition}\label{universal contraction} The operator system $\oss_1$ is the universal operator system of a contraction. That is, if
$\cl H$ is a Hilbert space and $T\in \B(\H)$ is a contraction, then there is a unital completely positive map $\phi: \oss_1
\to \B(\H)$ with $\phi(u_1) = T$. Conversely, if $\phi: \oss_1 \to \B(\H)$
is any unital completely positive map, then $\phi(u_1)$ is a contraction.
\end{proposition}

\begin{proof} If $T\in\B(\H)$ is a contraction, then $T$ has a unitary dilation $U\in\B(\K)$ for some Hilbert space $\K\supseteq\H$.
As $\cstar(\mathbb F_1)\equiv C(\mathbb T)$ is the universal C$^*$-algebra generated by a unitary, there is
a unital $*$-homomorphism $\pi:C(\mathbb T)\rightarrow\B(\K)$
such that $\pi(u_1)=U$. Let $\phi:\oss_1\rightarrow\B(\H)$ be the restriction of $\pi$ to $\oss_1$ followed by the compression of operators in $\B(\K)$ to $\H$.
Then $\phi$ is a unital completely positive map with $\phi(u_1) = T$. The converse follows from the fact that
unital completely positive maps are (completely) contractive.
\end{proof}

We now consider a dual result.

\begin{definition} The \emph{numerical radius} of an operator $T\in\B(\H)$ is the quantity $w(T)$ defined by
\[
w(T)\,=\,\sup\{|\langle T\xi,\xi\rangle|\,:\,\xi\in\H,\;\|\xi\|=1\}\,.
\]
\end{definition}

\begin{proposition}\label{universal num-rad contraction}
The operator system $\oss_1^d$ is the universal operator system of an operator with numerical radius at most $1/2$.
That is, if $\cl H$ is a Hilbert space and 
$T \in \B(\H)$ is such that $w(T) \le 1/2$, then there exists a unital
 completely positive map $\phi: \oss_1^d \to \B(\H)$ with $\phi(E_{12})
  = T$. Conversely, if $\phi: \oss_1^d \to \B(\H)$ is a unital completely
  positive map, then $w(\phi(E_{12})) \le 1/2$.
\end{proposition}

 \begin{proof} Given a unital completely positive map $\phi: \oss_1^d \to
   \B(\H)$, use Arveson's extension theorem to extend this map to a unital
   completely positive map $\phi:\M_2 \to \B(\H).$ Let $T = \phi(E_{12})$
   and set $A= \phi(E_{11}), B = \phi(E_{22}),$ so that $A+B =
   \phi(E_{11} +E_{22}) = I_\H$.
Since \[ X\,=\,\left[\begin{array}{cc} E_{11} & E_{12}\\E_{21}&
  E_{22} \end{array}\right] \] is positive, we have that
\begin{equation}\label{ando} \phi^{(2)}(X)\,=\,\left[\begin{array}{cc}  A &T\\T^* & B\end{array}\right] \,
\end{equation}
is positive.
By Ando's theorem \cite{ando1973}, $w(T) \le 1/2$.

Conversely, if $w(T) \le 1/2,$ then again by Ando's theorem there exist
$A,B \in \cl B(\cl H)_+$ with $A+B = I_\H$ such that the operator matrix \eqref{ando} is
positive. The linear map
$\phi: \M_2 \to \B(\H)$ given by
$\phi(E_{12}) = T = \phi(E_{21})^*$, $\phi(E_{11}) = A$ and $\phi(E_{22}) = B$ is completely positive by Choi's criterion
(see \cite[Theorem 3.14]{Paulsen-book}), and the proof is complete.
\end{proof}

An alternate proof of Proposition \ref{universal num-rad contraction} can be achieved by
appealing to Arveson's nilpotent dilation theorem \cite[Theorem 1.3.1]{arveson1972}
rather than Ando's theorem \cite{ando1973}.

\begin{theorem}\label{mixed}
The following tensor product relations involving $\oss_1$ and $\oss_1^d$ hold:
\begin{enumerate}
\item $\oss_1\omin\oss_1=\oss_1\oc\oss_1$;
\item $\oss_1\oc\oss_1\neq \oss_1\omax\oss_1$;
\item $\oss_1\omin\oss_1^d=\oss_1\oc\oss_1^d$;
\item $\oss_1\oc\oss_1^d\neq \oss_1\omax\oss_1^d$;
\item $\oss_1^d \omin \oss_1^d = \oss_1^d \oc \oss_1^d$;
\item $ \oss_1^d \oc \oss_1^d \ne \oss_1^d \omax \oss_1^d$.
\end{enumerate}
\end{theorem}

\begin{proof}
Statements (1) and (3) follow from Proposition \ref{c-min}, while (2) follows from Theorem \ref{sone}.

To prove (4), by \cite[Proposition 4.12]{kavruk2011}, we need only prove that $\oss_1^d$ is not completely order isomorphic to
a C$^*$-algebra. Because every finite-dimensional C$^*$-algebra is necessarily injective, we need only note that
$\oss_1^d$ is not an injective operator system. That $\oss_1^d$ fails to be injective is a consequence of \cite[Corollary 7.2]{choi--effros1977}.

For statement (5), let $\phi: \oss_1^d \to \cl B(\cl H)$ and $\psi: \oss_1^d \to \cl B(\cl H)$ be
unital completely positive maps with commuting ranges and let $T = \phi(E_{12}),$ $R= \psi(E_{12}).$ Then $w(T) \le 1/2,$ $w(R) \le 1/2,$ $RT=TR$ and $R^*T=TR^*.$
If we fix real numbers $0< r,t<1,$ then $w(tT) < 1/2$ and $w(rR) < 1/2.$
Let $\cstar(T)$ and $\wstar(T)$ denote the unital C$^*$-algebra and the von Neumann algebra generated by $T$, respectively.
By \cite[Theorem 1.1]{farenick--kavruk--paulsen2011}, there exist $A_t,B_t\in\cstar(T)_+$
such that $A_t+B_t = I$ and the matrix
\[ \begin{bmatrix} A_t & tT\\tT^* & B_t \end{bmatrix} \]
is positive.  Clearly, $A_t$ and $B_t$ commute with $R$ and $R^*.$  Now by taking limit points in the strong operator topology
we obtain operators $A$ and $B$ in $\wstar(T)$  such that $A+B = I$ and
\[ \begin{bmatrix} A & T\\T^* & B \end{bmatrix}\]
is positive.
Similarly, we obtain positive operators $C$ and $D$ in $\wstar(R)$ such that $C+D=I$ and
\[ \begin{bmatrix} C &R\\R^* & D \end{bmatrix} \] is positive.

By Choi's criterion \cite[Theorem 3.14]{Paulsen-book}, the maps $\tilde{\phi} : \cl M_2 \to \cl B(\cl H)$ and $\tilde{\psi} : \cl M_2 \to \cl B(\cl H)$
defined by $\tilde{\phi}(E_{11}) = A$, $\tilde{\phi}(E_{12}) = \tilde{\phi}(E_{21})^* = T$ , $\tilde{\phi}(E_{22}) = B$ and
$\tilde{\psi}(E_{11}) = C$, $\tilde{\psi}(E_{12}) = \tilde{\psi}(E_{21})^* = R$, $\tilde{\psi}(E_{22}) = D$
are completely positive extensions of $\phi$ and $\psi$, respectively, that have commuting ranges.
We thus obtain a unital completely positive map $\tilde{\phi} \otimes \tilde{\psi} : \cl M_2 \oc \cl M_2 = \cl M_2 \omin \cl M_2 \to \cl B(\cl H).$
Its restriction to $\oss_1^d \omin \oss_1^d \coisubset \cl M_2 \omin \cl M_2$ is clearly completely positive.

Since every pair of unital completely positive maps on $\cl S_1^d$ with commuting ranges
gives rise to a unital completely positive map on $\oss_1^d \omin \oss^d_1,$ we see that
$\cl M_n(\oss_1^d \omin \oss_1^d)_+ =\cl  M_n(\oss_1^d \oc \oss_1^d)_+$ for all $n.$  Hence,
$\oss_1^d \omin \oss_1^d = \oss_1 \oc \oss_1^d.$

To prove the final assertion, we use \cite[Proposition 1.9]{farenick--paulsen2011}:
$(\oss \omin \ost)^d = \oss^d \omax\ost^d$ and $(\oss \omax \ost)^d= \oss^d \omin \ost^d$ for all finite-dimensional operator systems
$\oss$ and $\ost$. Because $\oss_1 \omin \oss_1 \ne \oss_1 \omax \oss_1$, we conclude that
$\oss_1^d\omin \oss_1^d \ne \oss_1^d \omax \oss_1^d$. Hence, this inequality and $\oss_1^d \omin \oss_1^d = \oss_1^d \oc \oss_1^d$
imply that $\oss_1^d \oc \oss_1^d \ne \oss_1^d \omax \oss_1^d$.
\end{proof}

Recall that two operators $R$ and $S$ are $*$-commuting if $RS=SR$ and $RS^*=S^*R$.

\begin{corollary} The operator systems $\oss_1\omin\oss_1$ and $\oss_1^d\omin\oss_1^d$ are universal operator systems, respectively,
for  pairs of $*$-commuting contractions and pairs of $*$-commuting operators of numerical radius at most $\frac{1}{2}$.
\end{corollary}


\section{Tsirelson's Non-Commutative $n$-Cubes}\label{tsir}

Tsirelson \cite{tsirelson1980,tsirelson1993} studied certain operator systems that can
be best thought of as the non-commutative analogues of $n$-dimensional cubes.
To describe his ideas, we introduce an $(n+1)$-dimensional operator
system that we shall call the {\em operator system of the non-commuting n-cube,} and
denote it by $NC(n).$

\begin{definition} Let
 $\mathcal G=\{h_1,\dots,h_n\}$, let
$\mathcal R=\{h_j^*=h_j,\;\|h_j\|\leq 1,\;1\leq j\leq n\}$ be a set of relations in the set $\mathcal G$, and let $\cstar(\mathcal G\vert\mathcal R)$
denote the universal unital C$^*$-algebra generated by $\mathcal G$ subject to $\mathcal R$.
The operator system
\[
NC(n)\,=\,\mbox{\rm span} \{ 1, h_1,...,h_n \}\,\subset\, \cstar(\mathcal G\vert\mathcal R)\,.
\]
is called the \emph{operator system of the non-commuting $n$-cube}.
\end{definition}

By the universal property of universal C$^*$-algebras
we deduce immediately that the spectrum of each $h_i$ in  $\cstar(\mathcal G\vert\mathcal R)$ is $[-1,1]$
and that
for any $n$ hermitian contractions $A_1,\dots,A_n$ acting on any Hilbert space $\H$
there is a unital completely positive map $\phi: NC(n) \to \B(\H)$ with $\phi(h_i) = A_i$, $i = 1,\dots,n$.

If we had demanded that in addition the $A_i$'s pairwise commute, then for
our universal operator system we could have taken the span of the constant
function and the coordinate
functions $x_i, 1 \le i \le n$, inside the C$^*$-algebra $C([-1,1]^n)$ of continuous
functions on the $n$-dimensional cube. This motivates the following definition.

\begin{definition} The \emph{operator system of the commutative $n$-cube} is the operator subsystem
$C(n)\subset C([-1,1]^n)$ given by
\[
C(n) \,=\, \mbox{\rm span} \{1, x_1,...,x_n \} \,,
\]
where $x_i$ is the $i$th coordinate function on $[-1,1]^n$.
\end{definition}

The fact that $C(n)$ is the ``universal'' operator system for $n$ commuting selfadjoint contractions is shown
in the next proposition.

\begin{proposition} If $\H$ is a Hilbert space, $A_1,..., A_n \in
B(H)$ are any {\em pairwise commuting} operators satisfying $-I \le A_i \le I,$
then there is a unital completely positive map $\Phi:C(n) \to
\B(\H)$ with $\Phi(x_i) = A_i$, $i = 1,\dots,n$.
\end{proposition}
\begin{proof}
By the nuclearity of $C([-1,1])$, we have a natural identification
of $C([-1,1]^n)$ and the $n$-fold tensor product
$C([-1,1])\otimes_{\max} \dots \otimes_{\max} C([-1,1])$.
Let $\pi_i : C([-1,1])$ $\to$ $\cl B(\cl H)$ be the $*$-representation mapping the identity function to $A_i$, $i = 1,\dots,n$.
Since the ranges of $\pi_i$, $i = 1,\dots,n$, pairwise commute, there exists a $*$-representation
$\pi : C([-1,1]^n)\to \cl B(\cl H)$ such that $\pi(x_i) = A_i$, $i = 1,\dots,n$. It remains to let $\Phi$ be the
restriction of $\pi$ to $C(n)$.
\end{proof}

Our next proposition places $NC(n)$ into the setting of group operator systems described in Section \ref{s_osdg}.

\begin{notation} Let $*_n\mathbb Z_2=\mathbb Z_2*\cdots*\mathbb Z_2$ be the free product of $n$ copies of the group
$\mathbb Z_2$ of order two.
\end{notation}

\begin{proposition}\label{p_freep}
Let $u_i= u_i^* \in \cstar(*_n\mathbb Z_2)$ be the
  generator of the $i$-th copy of $\mathbb Z_2$ in $*_n\mathbb Z_2$
  and $\frak u_n = \{u_1,\dots,u_n\}$. Then the unital linear map $\Psi:
  \mathcal S(\frak u_n) \to NC(n)$ defined by $\Psi(u_i) = h_i, 1 \le i \le n$,
  is a complete order isomorphism.
\end{proposition}
\begin{proof} Since each $u_i$ is a self-adjoint unitary, it has
  spectrum $\{-1,1\}$ and hence $-1 \le u_i \le 1.$ Hence, by the
  universal property of $NC(n)$, the map $\Psi^{-1}$
  is completely positive. It hence suffices to
  show that $\Psi$ is completely positive. To prove this, it
  is enough to show that if $\H$ is any Hilbert space and $A_i \in
  \B(\H)$ satisfy $A_i = A_i^*$ and $-I \le A_i \le I,$ $1 \le i \le n$,
  then there is a unital completely positive map $\gamma: \mathcal
  S(\frak u_n) \to \B(\H)$ with $\gamma(u_i) = A_i$, $i = 1,\dots,n$.

With $A_i$ as above, let
\[ U_i \,=\, \left[\begin{matrix} A_i & (I - A_i^2)^{1/2}\\ (I - A_i^2)^{1/2}& -
  A_i \end{matrix}\right] \]
be the Halmos dilation of $A_i$. Then $U_i$ is a self-adjoint
unitary for each $i$. By the universal property of $\cstar(*_n\mathbb Z_2)$,
there is a unital $*$-homomorphism $\pi : \cstar(*_n\mathbb Z_2) \to \B(\H \oplus
\H)$ with $\pi(u_i) = U_i, 1 \le i \le n.$ Composing $\pi$ with the
compression onto the first copy of $\H$ yields a unital completely positive map
$\gamma: \cstar(*_n\mathbb Z_2) \to \B(\H)$ with $\gamma(u_i) = A_i,$ which completes
the proof.
\end{proof}

Propositions \ref{cstare of oss}  and \ref{p_freep} give the following immediate corollary.

\begin{corollary}\label{cenv nc cubes}
Up to a $*$-isomorphism, we have that $\cstare\left( NC(n)\right)=\cstar(*_n\mathbb Z_2)$.
\end{corollary}

There is at least one other representation of $NC(n)$ worth mentioning.

\begin{proposition}\label{NCinclusion}
Let $\{ u_1,\dots,u_n \}$ be the generators of $\oss_n.$
Then the linear map $\gamma: NC(n) \to \oss_n$ defined by $\gamma(h_i) = (u_i+u_i^*)/2$ is a complete order isomorphism onto its range,
which possesses a completely positive left inverse.
\end{proposition}
\begin{proof}
Since $-1 \le (u_i+u_i^*)/2 \le 1,$ the universal property of $NC(n)$ implies that the map $\gamma$ is unital and completely positive.
Conversely, since $\|h_i\| \le 1,$ the universal property of $\cl S_n$ implies that
there is a unital completely positive map $\psi: \oss_n \to NC(n)$ with $\psi(u_i) = h_i$, $i = 1,\dots,n$.
But then $\psi((u_i+u_i^*)/2) = h_i$, $i = 1,\dots,n$, and so $\psi$ is a left inverse of $\gamma.$ It follows that
$\gamma$ is a complete order isomorphism onto its range.
\end{proof}

It is easy to show that the spectrum of $(u_i+u_i^*)/2$ in $\cstar(\bb F_n)$ is $[-1,1]$.
However, by Corollary~5.6, the spectrum of $h_i$ in $\cstare(NC(n))$ is $\{ -1, 1\}$.
Hence, the C$^*$-algebra generated by the image of $NC(n)$ in $\cstar(\bb F_n)$ is much larger than its enveloping C$^*$-algebra.

\begin{proposition} Let $\mathbb Z_2 \oplus \cdots \oplus \mathbb Z_2$ be the direct sum of $n$ copies of the group of
  order two, let $v_i= v_i^* \in \cstar(\mathbb Z_2 \oplus
  \cdots \oplus \mathbb Z_2)$ be the
  generator of the $i$-th copy of $\mathbb Z_2$ and let $\mathcal S(\frak v_n) =
  \mbox{\rm span} \{1, v_i: 1 \le i \le n \}.$ Then the unital map $\Psi:
  \mathcal S(\frak v_n) \to C(n)$ defined by $\Psi(v_i) = x_i, 1 \le i \le n$,
  is a complete order isomorphism.
\end{proposition}
\begin{proof}
The proof proceeds as that of the previous result. Since
the elements $v_i$ are selfadjoint, pairwise commuting and satisfy $-1\le v_i \le 1,$ we have that
  $\Psi^{-1}$ is completely positive. To see that $\Psi$ is completely
  positive, use the Halmos dilation on the coordinate functions $x_i$
  to obtain {\em pairwise commuting} self-adjoint unitaries $V_i \in \M_2(C([0,1]^n)).$
The universal property of $\cstar(\mathbb Z_2 \oplus \cdots
\oplus \mathbb Z_2)$ shows that there exists a $*$-homomorphism $\pi: \cstar(\mathbb Z_2 \oplus \cdots
\oplus \mathbb Z_2) \to \M_2(C([0,1]^n))$ with $\pi(v_i) = V_i$, $i = 1,\dots,n$. Since
$\Psi$ is a compression of $\pi$, we have that $\Psi$ is completely positive.
\end{proof}

In \cite{kavruk2012}, the notion of a \emph{coproduct} of two operator systems was introduced.
If $\cl S$ and $\cl T$ are operator systems, the coproduct operator system $\cl S\oplus_1\cl T$ is
characterized by the following universal property: whenever $\cl U$ is an operator system and
$\phi : \cl S\to \cl U$ and $\psi : \cl T\to\cl U$ are unital completely positive maps, there exists a unique
unital completely positive map $\theta : \cl S\oplus_1 \cl T\to \cl U$ such that $\theta(\iota_{\cl S}(x)) = \phi(x)$
for all $x\in \cl S$ and $\theta(\iota_{\cl T}(y)) = \psi(y)$ for all $y\in \cl T$, where $\iota_{\cl S}$ and $\iota_{\cl T}$ are
the canonical inclusions of $\cl S$ and $\cl T$, respectively, into $\cl S\oplus_1\cl T$.
It can easily be checked that the coproduct is an associative operation.

On the other hand, the C$^*$-algebra $\cstar(*_n\mathbb Z_2)$ is canonically $*$-isomorphic to
the free product (involving $n$ terms) $\cstar(\mathbb{Z}_2)\ast\cdots\ast \cstar(\mathbb{Z}_2)$.
Since $\cstar(\mathbb{Z}_2)$ is $*$-isomorphic to $\ell^{\infty}_2$ via the Fourier transform,
it now follows from \cite[Proposition 4.3]{kavruk2012} and Proposition \ref{p_freep} that
$NC(n)$ is canonically order isomorphic to the coproduct (involving $n$ terms)
$\ell^{\infty}_2\oplus_1 \cdots\oplus_1 \ell^{\infty}_2$.

Another consequence of \cite[Theorem 4.8]{kavruk2012} is that $NC(2)$ is completely order isomorphic to
the quotient operator system $\ell^{\infty}/{{\rm Span}\{(1,1,-1,-1)\}}$. More precisely, let
$p_i =(1+u_i)/2, q_i = (1-u_i)/2,$ so that $p_i$ and $q_i$ are the spectral
projections for the self-adjoint unitary $u_i$  corresponding to the
eigenvalues $1$ and $-1,$ respectively, $i = 1,2$.
We have the following fact.

\begin{proposition}[Kavruk]\label{p_qk}
The linear map $\gamma: \ell^{\infty}_4 \to \mathcal
  S(\mathfrak u_2)$ defined by
\[\gamma((a_1,a_2,a_3,a_4)) = 1/2[a_1p_1+a_2q_1+a_3p_2+a_4q_2]\]
is a complete quotient map.
\end{proposition}

The following corollary describes strict positivity in $\M_n(NC(2))$.

\begin{corollary} Let $C_0, C_1, C_2 \in M_n$ be self-adjoint. Then
  $C_0 \otimes 1 + C_1 \otimes h_1 + C_2 \otimes h_2$ is strictly positive in $\M_n(NC(2))$
  if and only if there exists $A \in \M_n$ such that
$(C_0/2 \pm C_1  + A)$ and $(C_0/2 \pm C_2  - A) $ are strictly positive in $\M_n.$
\end{corollary}
\begin{proof}
By Proposition \ref{p_freep}, $NC(2) = \mathcal S(\mathfrak u_2)$. Letting $h_i =
  u_i,$ $i = 1,2$, we have by Proposition \ref{p_qk}
  that  $C_0 \otimes 1 + C_1 \otimes h_1 + C_2 \otimes h_2 $ is strictly positive  in $\M_n(NC(2))$
  if and only if
\begin{multline*} C_0 \otimes 1/2[p_1+q_1+p_2+q_2] + C_1 \otimes (p_1
  - q_1) + C_2 \otimes (p_2 - q_2) =\\ (C_0/2 +C_1) \otimes p_1 + (C_0/2 - C_1) \otimes q_1
  + (C_0/2 +C_2) \otimes p_2 +(C_0/2 - C_2) \otimes q_2
 \end{multline*}
 is strictly positive.
On the other hand, the latter term is strictly positive if and only if it has a
strictly positive pre-image in $\M_n(\ell^{\infty}_4).$ But any
pre-image must be of the described form for some $A$ in $\M_n.$
 \end{proof}

Recall that the (matrix ordered) dual of a finite dimensional operator system is again an operator system. Since a map
$\phi: \ell^{\infty}_n \to \cl M_p$ is completely positive if and only if $\phi(e_i) \ge 0$ for all $i,$ we see that the map 
that identifies $\phi$ with $(\phi(e_1),\dots,\phi(e_n)) \in \cl M_p(\ell^{\infty}_n)$ defines a complete order isomorphism 
between $(\ell^{\infty}_n)^d$ and $\ell^{\infty}_n.$ This is the context for the following result.

\begin{proposition}\label{NC(2)dual} The matrix ordered dual $NC(2)^d$ of $NC(2)$ is completely order isomorphic to the operator subsystem
\[\cl V = \{(a,b,c,d): a+b = c+d \} \subset \ell^{\infty}_4\]
via the map that sends a functional $f:NC(2) \to \bb C$ to the vector $(f(p_1),f(q_1)$, $f(p_2),f(q_2)).$
\end{proposition}
\begin{proof}
Since the map $\gamma: \ell^{\infty}_4 \to NC(2)$ is a complete quotient map, 
the adjoint map $\gamma^d: NC(2)^d \to (\ell^{\infty}_4)^d = \ell^{\infty}_4$ is a complete order inclusion \cite[Proposition 1.8]{farenick--paulsen2011}. 
It is easy to verify that
$\cl V$ is the range of $\gamma^d$ and that $\gamma^d(f) = (f(p_1),f(q_1),f(p_2),f(q_2)).$
\end{proof}

\begin{remark} A class of \emph{operator spaces} denoted $NSG(n,k)$, $n,k\in \mathbb{N}$, was considered
in \cite{junge_etal}. It is not difficult to see that $NSG(n,2)$
coincides with the operator space dual of $NC(n)$.
However, since these objects were studied as operator spaces, their operator system structure
was not discussed in detail. Our emphasis, on the other hand, is on the
operator system tensor product properties of non-commutative $n$-cubes
and for this we need characterizations of the matrix ordered duals of
the \emph{operator system} $NC(n).$ This will be
fully developed in the next section.
\end{remark}

We conclude this section with a description of the matrix ordered dual $NC(n)^d$ for any $n\geq 2$, in the spirit of \cite{farenick--paulsen2011}.

\begin{proposition}\label{p_newd}
The matrix ordered dual $NC(n)^d$ of $NC(n)$ is completely order isomorphic to
the operator subsystem $\cl W$ of $\oplus_{k=1}^n\cl M_2$ given by
$$\cl W = \left\{\bigoplus_{k=1}^n \left(\smallmatrix a_{11}^k & a_{12}^k\\ a_{21}^k & a_{22}^k\endsmallmatrix\right) : a_{ii}^k = a_{jj}^l,
a_{12}^k = a_{21}^k, i,j=1,2, k,l = 1,\dots,n\right\}.$$
\end{proposition}
\begin{proof}
Recall the completely positive maps $\gamma: NC(n)\rightarrow\oss_n$ and $\psi : \cl S_n\to NC(n)$ from Proposition \ref{NCinclusion},
given by $\gamma(h_i)=(u_i+u_i^*)/2$ and $\psi(u_i) = h_i$, $i = 1,\dots,n$.  If
$Y\in\M_p(NC(n))$ is strictly positive, then so is $X=\gamma^{(p)}(Y)\in\M_p(\oss_n)$. Recall that $\psi$ is a left inverse of $\gamma$.
Thus, $Y=\psi^{(p)}(X)$ , which shows that $\psi$ is a complete quotient map (Proposition \ref{quo map criterion}), and so 
$\psi^d$ is a completely order embedding of $NC(n)^d$ into $\cl S_n^d$.

Let $\cl U\subseteq \oplus_{k=1}^n\cl M_2$ be the operator system given by
$$\cl U = \left\{\bigoplus_{k=1}^n \left(\smallmatrix a_{11}^k & a_{12}^k\\ a_{21}^k & a_{22}^k\endsmallmatrix\right) : a_{ii}^k = a_{jj}^l,
i,j=1,2, k,l = 1,\dots,n\right\}$$ and
let $\theta : \cl S_n^d \to \cl U$ be the complete order isomorphism from \cite[Theorem 4.4]{farenick--paulsen2011}.
Let $\{F_j\}_{j=-n}^n$ be the basis of $\cl S_n^d$, dual to the basis $\{u_j\}_{j=-n}^n$.
It follows from the proof of \cite[Theorem 4.4]{farenick--paulsen2011} that
$\theta$ maps $F_0$ to $I$ and $F_j$ to $\oplus_{k=1}^n A_k$, where $A_k = 0$ if $k\neq j$ and
$A_j = \left(\smallmatrix 0 & 1/(n+1)\\ 0 & 0\endsmallmatrix\right)$, $j = 1,\dots,n$.
The composition $\theta\circ \psi^d : NC(n)^d \to \cl U$ is a complete order embedding.
Let $\{H_0,H_1,\dots,H_n\}$ be the basis of $NC(n)^d$, dual to the canonical basis
$\{1,h_1,\dots,h_n\}$ of $NC(n)$.
From the definition of $\psi$, we have that $\psi^d(H_i) = F_i + F_{-i}$, $i = 1,\dots,n$.
It follows that the range of $\theta\circ \psi^d$ is $\cl W$.
\end{proof}


\section{Tensor Products of Non-commutative Cubes}\label{s_tpnc}

Although not stated in the language of operator systems, the
calculations in \cite{tsirelson1980,tsirelson1993} of various
bipartite correlation boxes amount to the
calculation of various tensor products on $NC(n) \otimes NC(m).$
However, our general theory of operator system tensor products and our
duality results give an alternate approach to this
theory. We give our own derivation of these tensor
results in the present section, while in the next section we introduce bipartite correlation
boxes and explain how to translate results between the two theories.
We begin with a general nuclearity result.

\begin{proposition} Let $\cl R$ be an operator system. Then $\cl R \omin NC(1) = \cl R \omax NC(1).$
\end{proposition}

An indirect way to see this statement is to use the fact that
$NC(1)$ is completely order isomorphic to $\ell^{\infty}_2$, which is a nuclear C$^*$-algebra.
However, our results yield a very short direct proof.

\begin{proof} We have that $\cl M_n(\cl R \omax NC(1))_+ \coisubset \cl M_n(\cl R \omin NC(1))_+,$ so it will be sufficient to prove the
converse inclusion. The identification $\cl M_n(\cl R \omin \cl S)_+ = (\cl M_n(\cl R) \omin \cl S)_+$ shows that it suffices to consider the case $n=1.$

Given $r_1, r_2 \in \cl R,$ we have that $r_1 \otimes 1 + r_2 \otimes h_1 \in [\cl R \omin NC(1)]_+$ if and only if $r_1 \pm r_2 \in \cl R_+.$
Since  $r_1 + r_2$ and $r_1 - r_2$ are in $\cl R_+$ and
$1+h_1$ and $1 - h_1$
are in $NC(1)_+,$ it follows that
\[2[r_1 \otimes 1 + r_2 \otimes h_1] =(r_1 - r_2) \otimes (1- h_1) + (r_1 + r_2) \otimes (1 + h_1) \in (\cl R \omax NC(1))_+ \]
and the proof is complete.
\end{proof}

The following result shows that the operator system $NC(n)$ satisfies the hypotheses of Theorem \ref{l_uv}.

\begin{lemma}\label{nc cubes}
For every operator system $\cl R$ we have that $\cl R\oc NC(n)\coisubset \cl R\omax \cstar(*_n \bb Z_2)$.
Thus, $NC(m)\oc NC(n)\coisubset \cstar(*_m\mathbb
Z_2)\omax\cstar(*_n\mathbb Z_2)$ and therefore $NC(m) \oc NC(n) = NC(m) \oess NC(n).$
\end{lemma}
\begin{proof} The proof follows closely the argument given in \cite[Lemma 4.1]{farenick--kavruk--paulsen2011}.
It suffices to show that if $\phi : \cl R\rightarrow\cl B(\H)$ and $\psi : NC(n)\rightarrow \cl B(\H)$ are unital completely positive maps with commuting ranges then
$\psi$ can be extended to a (unital) completely positive map $\tilde{\psi} : \cstar(*_n \bb Z_2)\rightarrow \cl B(\H)$ whose range commutes with the range of $\phi$.
This is done similarly to \cite[Lemma 4.1]{farenick--kavruk--paulsen2011}: we have that
$h_i$ is a selfadjoint unitary, and hence $\psi(h_i)$ is a selfadjoint operator with $-I \leq \psi(h_i)\leq I$. Write
$$w_i = \left[\begin{matrix} \psi(h_i) & (1 - \psi(h_i))^{1/2}\\  (1 - \psi(h_i))^{1/2} & -\psi(h_i)\end{matrix}\right].$$
Then $w_i$ is a selafdjoint unitary, and by the universal property of $\cstar(*_n \bb Z_2)$, there exists a $*$-homomorphism
$\pi : \cstar(*_n \bb Z_2) \rightarrow \cl B(\H\oplus \H)$ such that $\pi(h_i) = w_i$. Letting $\tilde{\phi}(r) = \left[\smallmatrix \phi(r) & 0\\ 0 & \phi(r)\endsmallmatrix\right]$,
we conclude that the ranges of $\pi$ and $\tilde{\psi}$ commute. Now letting $\tilde{\psi}(x) = p\pi(x)p$, where
$p = \left[\smallmatrix I & 0\\ 0 & 0\endsmallmatrix\right]$, we conclude that the ranges of $\tilde{\psi}$ and $\phi$ commute and
since $\tilde{\psi}$ extends $\psi$, we have shown the first claim.

The second claim follows by applying the first claim twice and using the identification $\cl R \oc \cl S = \cl S \oc \cl R.$
\end{proof}

\begin{theorem}\label{th_nc2}
For every operator system $\cl R$ we have that $\cl R \omin NC(2) = \cl R \oc NC(2).$
\end{theorem}
\begin{proof} It is a rather well-known fact that
the group $\bb Z_2 * \bb Z_2$ is amenable (see Remark \ref{r_ame}).
Hence, $\cstar(\bb Z_2 * \bb Z_2)$ is nuclear and it follows that $\cl R \omin \cstar(\bb Z_2 * \bb Z_2) = \cl R \omax \cstar(\bb Z_2 * \bb Z_2).$
Since $\cl R \omin NC(2) \coisubset \cl R \omin \cstar(\bb Z_2* \bb Z_2)$ and,
by Lemma \ref{nc cubes}, $\cl R \oc NC(2) \coisubset \cl R \omax \cstar(\bb Z_2 * \bb Z_2)$, the result follows.
\end{proof}

\begin{remark}\label{r_ame}
One way to see that $\bb Z_2 * \bb Z_2$ is amenable is to note that the subgroup $H= \{ (g_1g_2)^n: n \in \bb Z \}$ is isomorphic to $\bb Z$ and $(g_1g_2)^{-1} = g_2 g_1.$ If we let $\bb Z_2$ act on $H$ by the idempotent automorphism of conjugation by $g_2,$ then $\bb Z_2 * \bb Z_2$ is seen to be the semidirect product of $H$ by $\bb Z_2.$
Now use the fact the semidirect products of amenable groups are amenable.
\end{remark}

The following computational lemmas about the behaviour of $\max$ will be useful below.

\begin{lemma}\label{l_compu} Let $\cl R$ and $\cl S$ be finite dimensional vector spaces,
$\{r_1,...,r_m \}\subseteq \cl R$ and $\{s_1,...,s_n \}\subseteq \cl S$.
Let $(x_{i,j}) \in \cl M_p(\cl R)$ and $(y_{i,j}) \in \cl M_p(\cl S)$ and assume that
$(x_{i,j}) = \sum_{k=1}^m A_k \otimes r_k$ and $(y_{i,j}) = \sum_{l=1}^n B_l \otimes s_l,$ where $A_k,B_l\in \cl M_p$,
$k = 1,\dots,m$, $l = 1,\dots,n$. Then
\[ \sum_{i,j=1}^n x_{i,j} \otimes y_{i,j} = \sum_{k=1}^m \sum_{l=1}^n Tr(A_kB_l^t) r_k \otimes s_l ,\]
where $Tr$ denotes the unnormalised trace.
\end{lemma}
\begin{proof} By the bilinearlity of the tensor products it is enough to consider the case where
$(x_{i,j}) = (a_{i,j} r_k)$ and $(y_{i,j}) = (b_{i,j} s_l)$.
But, in this case, $\sum_{i,j=1}^p (a_{i,j} r_k) \otimes (b_{i,j} s_l) = Tr(AB^t)r_k \otimes s_l$,
and the proof is complete.
\end{proof}

\begin{lemma}\label{l_spm}
Let $\cl R$ and $\cl S$ be finite dimensional operator systems,
$\{r_1,...,r_m \}\subseteq \cl R$ and $\{s_1,...,s_n \}\subseteq \cl S$ be
linear bases,
and $u = \sum_{k=1}^m \sum_{l=1}^n x_{k,l} r_k \otimes s_l \in \cl R\otimes\cl S$.
If $u$ is strictly positive in $(\cl R \omax \cl S)_+$, then there exist $p\in\mathbb{N}$ and elements
$U_1 = \sum_{k=1}^m A_k \otimes r_k \in M_p(\cl R)_+$ and $U_2 = \sum_{l=1}^n B_l \otimes s_l \in M_p(\cl S)_+$ such that
$Tr(A_kB_l^t) = x_{k,l}$, $1 \le k \le m, 1 \le l \le n.$
\end{lemma}
\begin{proof}
Suppose that $u$ is strictly positive. By Lemma \ref{f-dim max}, there exist $p\in \mathbb{N}$, $U_1 = (\alpha_{i,j})\in M_p(\cl R)^+$ and
$U_2 = (\beta_{i,j})\in M_p(\cl S)^+$ such that $u = \sum_{i,j = 1}^p \alpha_{i,j}\otimes\beta_{i,j}$.
Write $U_1 = \sum_{k=1}^m A_k \otimes r_k$ and $U_2 = \sum_{l=1}^n B_l \otimes s_l$ and use Lemma \ref{l_compu}
to obtain the desired form of $u$.
\end{proof}

\begin{remark} It is not difficult to see that 
every element $u$ of the form prescribed in Lemma \ref{l_spm} is necessarily positive;
however, this fact will not be needed in the sequel.
\end{remark}

We recall the operator system $\cl V$ from Section \ref{tsir}:
$$\cl V = \{(a,b,c,d)\in \ell^{\infty}_4 : a + b = c + d\}\subseteq \ell^{\infty}_4.$$
If $p\in \mathbb{N}$ then
$$\cl M_p(\cl V) = \left\{\sum_{i=1}^4 X_i \otimes e_i : X_i\in \cl M_p(\mathbb{C}), i = 1,2,3,4, X_1 + X_2 = X_3 + X_4\right\}.$$
Moreover,
$$\cl M_p(\cl V)_+ = \left\{\sum_{i=1}^4 X_i \otimes e_i \in \cl M_p(\cl V) : X_i\in \cl M_p(\mathbb{C})_+, i = 1,2,3,4\right\}.$$


The following result gives a more concrete representation of the
strictly positive elements of $\cl V \omax \cl V.$

\begin{proposition} Let $u = \sum_{i,j=1}^4 q_{i,j} e_i \otimes e_j$
  be a strictly positive element of $\cl V \omax \cl V.$ Then there
  exists $p$ and matrices $X_i,Y_j \in \cl M_p(\mathbb{C})_+$ with $X_1+X_2= X_3+X_4$
  and $Y_1+Y_2 = Y_3+Y_4 =I$ such that $q_{i,j} = Tr(X_iY_j)$, $i,j = 1,2,3,4$.
\end{proposition}
\begin{proof}
 By Lemma \ref{l_spm}, there exist $p\in \mathbb{N}$ and
elements $U_1 = \sum_{i=1}^4X_i \otimes e_i$ and $U_2 = \sum_{j=1}^4Y_j^t \otimes e_j$ in $\cl M_p(\cl V)_+$ such that
$\sum_{i,j=1}^4  Tr(X_iY_j^t)e_i \otimes e_j = \sum_{i,j=1}^4 q_{i,j} e_i \otimes e_j.$
The fact that $U_1,U_2\in \cl M_p(\cl V)_+$ implies moreover that $X_i,
Y_j \in \cl M_p(\mathbb{C})_+$ and $X_1+ X_2= X_3 + X_4,$ $Y_1+ Y_2= Y_3+ Y_4.$ Since
$Y_i \in \cl M_p(\mathbb{C})_+$ if and only if $Y_i^t \in \cl M_p(\mathbb{C})_+,$ we may and do
replace $Y_i$ by $Y_i^t.$

We thus have $q_{i,j} = Tr(X_iY_j)$ with
$X_i,Y_j \in \cl M_p(\mathbb{C})_+$ and $X_1+X_2=X_3+X_4$,
$Y_1+Y_2= Y_3+Y_4.$

Let $P$ be a positive invertible matrix such that
$\sum_{j=1}^4 P^{-1}Y_jP^{-1} =2E$ where $E$ is the
projection onto ${\rm ker}(Y_1+Y_2)^{\perp}$ and set $\hat{Y_i} =
P^{-1}Y_iP^{-1}, \hat{X_i} = PX_iP.$ We have that
$\hat{Y_1} + \hat{Y_2} = \hat{Y_3} + \hat{Y_4} = E$ and
$Tr(X_iY_j) = Tr(\hat{X_i}\hat{Y_j})= Tr((E\hat{X_i}E)(E\hat{Y_j}E))$.
After replacing $\hat{X_i}$ and $\hat{Y_j}$ with
$E\hat{X_i}E$ and $E\hat{Y_j}E$, respectively, the same equations will hold.  Diagonalizing $E,$
we may regard the matrices $E\hat{X_i}E$ and $E\hat{Y_j}E$ as matrices
of a smaller size. So no generality is lost
in assuming that $\sum_{j=1}^4 Y_j = 2I$ or, equivalently, that $Y_1+Y_2 = Y_3+Y_4 = I.$
\end{proof}

\begin{lemma}\label{th_ineq}
Let $p\in \mathbb{N}$ and $X_i, Y_j\in \cl M_p(\mathbb{C})_+$, $i,j = 1,2,3,4$, with $X_1 + X_2 = X_3 + X_4$ and $Y_1 + Y_2 = Y_3 + Y_4 = I$.
Set $q_{i,j} = Tr(X_iY_j)$, $i,j = 1,2$. For $a,c \in \{0,2\}$, set
$S_{a,c}(j,k) = \min_{b = 0,2}\{ \sum_{i=1}^2
\sqrt{q_{b+i,a+j}}\sqrt{q_{b+i,c+k}} \}.$ Then the following inequality holds:
$$Tr(X_1 + X_2) \leq \min_{a,c} \sum_{j,k = 1}^2 S_{a,c}(j,k).$$
\end{lemma}
\begin{proof}
Set $q = Tr(X_1 + X_2)$ for brevity.
For all $a,b,c\in \{0,2\}$ we have
\begin{eqnarray*}
q & = & Tr((Y_{a+1} + Y_{a+2})(X_{b+1} + X_{b+2})(Y_{c+1} + Y_{c+2}))\\
& = & \sum_{i,j,k = 1}^2 Tr(Y_{a+j}X_{b+i} Y_{c+k}) =
\sum_{i,j,k = 1}^2 Tr((Y_{a+j}X_{b+i}^{1/2})(X_{b+i}^{1/2} Y_{c+k}))\\
& = &  \sum_{i,j,k = 1}^2 Tr((X_{b+i}^{1/2}Y_{a+j})^* (X_{b+i}^{1/2} Y_{c+k}))
\leq \sum_{j,k = 1}^2 \sum_{i=1}^2\|X_{b+i}^{1/2}Y_{a+j}\|_2 \|X_{b+i}^{1/2} Y_{c+k}\|_2.
\end{eqnarray*}
On the other hand, if $X$ and $Y$ are positive matrices with $Y\leq I$, then
$$\|X^{1/2}Y\|_2^2 = Tr(X^{1/2}Y^2X^{1/2}) \leq Tr(X^{1/2}Y X^{1/2}) = \|X^{1/2}Y^{1/2}\|_2^2.$$
It follows that
$$q\leq \sum_{i,j,k = 1}^2 \|X_{b+i}^{1/2}Y_{a+j}^{1/2}\|_2 \|X_{b+i}^{1/2} Y_{c+k}^{1/2}\|_2.$$
Since $\|X^{1/2} Y^{1/2}\|_2^2 = Tr(Y^{1/2}XY^{1/2}) = Tr(XY)$ whenever $X$ and $Y$ are positive matrices, we conclude that
$$q\leq \sum_{j,k = 1}^2 \sum_{i=1}^2 \sqrt{q_{b+i,a+j}}\sqrt{q_{b+i,c+k}}$$ whenever $b\in \{0,2\}$.
\end{proof}

Lemma \ref{th_ineq} gives us a Bell type inequality. It has a form
similar to the CHSH inequality of \cite{CHSH}, except the appearance
of square roots seems to be new.

\begin{theorem}\label{th_bi} 
Let $u = \sum_{i,j=1}^4 q_{i,j} e_i \otimes e_j \in
  (\cl V \omax \cl V)_+.$
 For $a,c \in \{0,2\}$, set
$S_{a,c}(j,k) = \min \{ \sum_{i=1}^2
\sqrt{q_{b+i,a+j}}\sqrt{q_{b+i,c+k}} \}.$ Then for $d \in \{0,2\}$ the following inequality holds:
$$ \sum_{i,j=1}^2 q_{d+i,d+j} \leq \min_{a,c} \sum_{j,k = 1}^2 S_{a,c}(j,k).$$
\end{theorem}
\begin{proof} For any $\delta >0,$ the element
\[u + \delta 1 \otimes 1 = \sum_{i,j=1}^2 (q_{i,j}+\delta) e_i \otimes
  e_j\]
is strictly positive; by Lemma \ref{th_ineq}, there exist $p\in \mathbb{N}$ and matrices $X_i, Y_j\in \cl M_p(\mathbb{C})_+$
such that $X_1 + X_2 = X_3 + X_4$, $Y_1 + Y_2 = Y_3 + Y_4 = I$ and $q_{i,j} + \delta = Tr(X_iY_j).$ The result
now follows by observing that, for $d \in \{0,2\},$
\[ Tr(X_1+X_2) = Tr((X_{d+1}+X_{d+2})(Y_{d+1} + Y_{d+2})) = \sum_{i,j=1}^2 (q_{d+i,d+j}+ \delta)\]
and letting $\delta \to 0.$
\end{proof}

\begin{theorem}\label{th_v}
$\cl V \omin \cl V \ne \cl V \omax \cl V.$
\end{theorem}
\begin{proof}
If we identify $\ell^{\infty}_4 \otimes \ell^{\infty}_4$ with $4 \times 4$ matrices by the map $e_i \otimes e_j \to E_{i,j},$ then
$\cl V \otimes \cl V$ is identified with the $4 \times 4$ matrices such that:
\begin{enumerate}
\item the first two terms in each row has the same sum as the last two terms;
\item the first two terms in each column has the same sum as the last two terms.
\end{enumerate}

Let
\[Q = [q_{ij}]_{i,j=1}^4 = \begin{bmatrix} 1& 0& 1 & 0\\ 0 & 1 & 0 & 1\\1 & 0 & 1 & 0\\0 & 1 & 0 & 1 \end{bmatrix}; \]
clearly, $Q$ is in $(\cl V \omin \cl V)_+ = (\cl V \otimes \cl V) \cap (\ell^{\infty}_4 \omin \ell^{\infty}_4)_+.$
The proof will be complete if we show that $Q\not\in (\cl V \otimes_{\max} \cl V)_+.$

We have that $\sum_{i,j=1}^2 q_{d+i,d+j} = 2$, for $d\in \{0,2\}$. 
Now set $a=0, c=2.$
Taking $b=0,$ we see that
\[0 \le S_{0,2}(2,1) \le \sum_{i=1}^2 \sqrt{q_{i,2}} \sqrt{q_{i,3}} =
0 \text{ and }
0 \le S_{0,2}(1,2) \le \sum_{i=1}^2 \sqrt{q_{i,1}}\sqrt{q_{i,4}} =0.\]
On the other hand, taking $b=2$ yields
\[0 \le S_{0,2}(1,1) \le \sum_{i=1}^2 \sqrt{q_{2+i,1}}\sqrt{q_{2+i,3}}
=0 \text{ and }
0 \le S_{0,2}(2,2) \le \sum_{i=1}^2 \sqrt{q_{2+i,2}}\sqrt{q_{2+i,4}} =0.\]
This violates the inequalities of Theorem \ref{th_bi}.
\end{proof}

\begin{corollary}\label{c_nc}
$NC(2) \oc NC(2) \ne NC(2) \omax NC(2).$
\end{corollary}
\begin{proof}
By Proposition~\ref{NC(2)dual}, $NC(2)^d = \cl V$.
Assume that $NC(2) \omin NC(2) = NC(2) \omax NC(2)$; by taking duals, we have
$NC(2)^d \omax NC(2)^d = NC(2)^d \omin NC(2)^d$, which contradicts Theorem \ref{th_v}. The proof is completed by using the fact that $NC(2) \omin NC(2) = NC(2) \oc NC(2)$
(see Theorem \ref{th_nc2}).
\end{proof}

\begin{theorem}[Tsirelson]\label{t-thm} For every $n,m \ge 2,$ $NC(n) \oc NC(m) \ne NC(n) \omax NC(m).$
\end{theorem}
\begin{proof} This follows from the $n=2$ case, similarly to the proof of Theorem \ref{ocomax} and using 
the fact that there exists a completely order isomorphic inclusion $\iota: NC(2) \to NC(k)$ and a unital completely positive map 
$\psi:NC(k) \to NC(2)$ whose composition is the identity on $NC(2)$.
\end{proof}

\begin{remark} This last result gives an alternate proof of a weaker result thanTheorem~\ref{ocomax}. Namely, it shows that for every $n,m \ge 2,$ $\oss_n \oc \oss_m \ne \oss_n \omax \oss_m.$
To see this, let $\gamma_n$ and $\psi_n$ be the maps defined in Proposition~\ref{NCinclusion}.  
Assume that $\oss_n \oc \oss_m = \oss_n \omax \oss_m.$ By functoriality, the maps 
$\gamma_n \oc \gamma_m: NC(n) \oc NC(m) \to \oss_n \oc \oss_m$ and 
$\psi_n \omax \psi_m: \oss_n \omax \oss_m \to NC(n) \omax NC(m)$ are unital and completely positive. 
Thus,
\[(\psi_n \omax \psi_m) \circ (\gamma_n \oc \gamma_m): NC(n) \oc NC(m) \to NC(n) \omax NC(m),\]
is a unital completely positive map and hence $NC(n) \oc NC(m) = NC(n) \omax NC(m),$ a contradiction.
\end{remark}

We record the following corollary whose proof follows closely that of Theorem \ref{omaxe}.

\begin{corollary}\label{omaxe2}
We have that $\cstare(NC(n)\omax NC(m)) \neq \cstare(NC(n))\omax\cstare(NC(m))$ for every $n,m\geq 2$.
\end{corollary}

\begin{remark} 
The fact that $\bb F_2$ embeds in $\bb Z_2 * \bb Z_2 * \bb Z_2$ and
the technique of the second author from \cite[Theorem~5.3]{kavruk2011}
shows that Kirchberg's Conjecture is equivalent to 
the identity $\cl S(\mathfrak{u}_3) \omin \cl S(\mathfrak{u}_3) =
\cl S(\mathfrak{u}_3) \oc \cl S(\mathfrak{u}_3)$ or, equivalently, to the identity 
$NC(3) \omin NC(3) = NC(3) \oc NC(3).$ So determining further relations in this direction will be quite difficult. 
Tsirelson~\cite{tsirelson1980,tsirelson1993} makes some claims that, if true, would imply that $[NC(m) \omin NC(n)]_+ = [NC(n) \oc NC(m)]_+,$ i.e., that these operator systems are equal
at the ground level. 
However, one step in his proof remains unjustified and was later posted as a problem. See Fritz' paper \cite{fritz2012} for a discussion.
\end{remark}

It is natural to wonder about some other operator system tensor products. In particular \cite{kavruk--paulsen--todorov--tomforde2011} 
introduces three other tensor products that lie between $\omin$ and $\oc.$ If one assumes that $\oss \coisubset \B(\H)$ and $\ost \coisubset \B(\K)$, then these are given by the identifications,
\[ \oss \otimes_{\rm el} \ost \coisubset \B(\H) \omax \ost, \quad \oss \otimes_{\rm er} \ost \coisubset \oss \omax \B(\K),
\]
and
 \[
\oss \otimes_{\rm e} \ost \coisubset \B(\H) \omax \B(\K).
\]

\begin{proposition} Let $n,m \in \bb N,$ then $NC(n) \omin NC(m) =
  NC(n) \otimes_{\rm el} NC(m) = NC(n) \otimes_{\rm er} NC(m).$ 
\end{proposition}
\begin{proof} It is easily checked that $NC(n)$ and $NC(m)$ have the
  OSLLP property of
  \cite[Definition~8.2]{kavruk--paulsen--todorov--tomforde2010}. Hence,
  by Theorem~8.1 and Theorem~8.5 of
  \cite{kavruk--paulsen--todorov--tomforde2010} together with the fact
  that the el and er tensors are identical modulo the
  flip, we have the claimed equalities.
\end{proof}

\begin{question}\label{nc-q} This leads us to ask:
\begin{enumerate}
\item Are any of the operator systems $NC(m) \omin NC(n)$, $NC(m) \otimes_{\rm e} NC(n)$, and $NC(m) \oc NC(n)$ equal?
\item What about the $C(m) \otimes C(n)$ cases?
\end{enumerate}
\end{question}


\section{Bipartite correlation boxes}\label{bcb}

In this section, we identify and discuss the relation of our results from Section \ref{s_tpnc} with quantum correlations
studied in \cite{barrett}, \cite{fritz2012}, \cite{tsirelson1980}, \cite{tsirelson1993}, among others.
Suppose that Alice and Bob perform an experiment in which Alice is given an input value $x$ and produces an output value $a$,
while Bob is given an input value $y$ and produces an output value $b$. We assume that the possible values of the $x,y,a,b$
are $0$ and $1$.
Let $p^1_{a|x}$ be the probability that Alice returns the value $a$ provided she is given the input $x$; similarly, let
$p^2_{b|y}$ be the probability that Bob returns the value $b$ provided he is given the input $y$.
These probabilities satisfy the following standard conditions: $p^1_{a|x} \geq 0$, $p^2_{b|y} \geq 0$, for all $a,b,x,y\in \{0,1\}$,
$p^1_{0|x} + p^1_{1|x} = 1$ for $x = 0,1$, and
$p^2_{0|y} + p^2_{1|y} = 1$ for $y = 0,1$.

It is clear that the family $(p^1_{0|0}, p^1_{1|0}, p^1_{0|1}, p^1_{1|1})$ is an element of the positive cone of the
operator system $\cl V$ defined in Section \ref{tsir}; moreover, every element of $\cl V_+$,
after normalisation, can be written in such a form. We similarly have that
$(p^2_{0|0}, p^2_{1|0}, p^2_{0|1}, p^2_{1|1})\in \cl V_+$.

Let $p_{a,b|x,y}$ be the probability that the pair $(a,b)$ is produced as an output by Alice and Bob,
provided that Alice is given an input $x$ and Bob is given an input $y$.
A \emph{bipartite correlation box} (which will be simply referred to by a \emph{box})
is a table of probabilities of the form $(p_{a,b|x,y})_{a,b,x,y}$, viewed as an element of $\ell^{\infty}_{16}$.
The positivity conditions $p_{a,b|x,y} \geq 0$, $a,b,x,y\in \{0,1\}$ are supposed to hold, as
is the normalisation condition $\sum_{a,b = 0}^1 p_{a,b|x,y} = 1$, $x,y\in \{0,1\}$.
In a \lq\lq non-signaling'' experiment, Alice and Bob are \lq\lq not allowed to communicate,'' which, in terms of the
probability table, is expressed by requiring that
$$p_{a,0|x,0} + p_{a,1|x,0} = p_{a,0|x,1} + p_{a,1|x,1} = p^1_{a|x}, \mbox{ for all } a,x\in \{0,1\},$$
$$p_{0,b|0,y} + p_{1,b|0,y} = p_{0,b|1,y} + p_{1,b|1,y} = p^2_{b|y}, \mbox{ for all } b,y \in \{0,1\}.$$
We will assume that all boxes represent probability distributions of non-signaling experiments.

A box $(p_{a,b|x,y})_{a,b,x,y}$ is called \emph{local} if
there exists a probability distribution $(r(\lambda))_{\lambda}$ (that is,
a finite family $(r(\lambda))_{\lambda}$ of non-negative real numbers with
$\sum_{\lambda} r(\lambda) = 1$) and, for each $\lambda$,
elements
$p^k(\lambda) = (p^k_{0|0}(\lambda),p^k_{1|0}(\lambda),p^k_{0|1}(\lambda),p^k_{1|1}(\lambda))\in \cl V_+$, $k = 1,2$, normalised so that
$p^k_{0|0}(\lambda) + p^k_{1|0}(\lambda) = 1$, $k = 1,2$, such that
$$p_{a,b|x,y} = \sum_{\lambda} r(\lambda) p^1_{a|x}(\lambda) p^2_{b|y}(\lambda), \ \ \ a,b,x,y\in \{0,1\}.$$

Tsirelson~\cite{tsirelson1980} introduced \emph{quantum} correlation boxes. These are the
probability distributions $(p_{a,b|x,y})$ given by
$p_{a,b|x,y} = Tr(\rho(A_x^a\otimes A_y^b))$, where $A_x^a$ and $A_x^b$ are positive operators
acting on corresponding Hilbert spaces $\H_x$ and $\H_y$
such that $A_x^0 + A_x^1 = I$ and $A_y^0 + A_y^1 = I$ for all $x,y\in \{0,1\}$, and $\rho$ is a
positive trace-class operator of unit trace. Tsirelson showed that these operators can be taken to act on a Hilbert space of dimension $2$.

Following \cite{barrett}, we let $\cl P$ be the set of all correlation boxes, $\cl L$ be the closure of 
the set of all local correlation boxes, and $\cl Q$ be the closure of 
the set of all quantum correlation boxes. Clearly, $\cl L\subseteq \cl Q \subseteq \cl P$ and each of these sets is convex.

In the sequel, we also identify the linear space $\cl V\otimes\cl V$ with $\cl M_4$ by the mapping sending $e_i\otimes e_j$ to $E_{ij}$.
We denote by $BS_4$ the set of all bistochastic matrices in $\cl M_4$, that is,
$$BS_4 = \left\{(a_{i,j})\in \cl M_4 : a_{k,l}\geq 0, \sum_{j=1}^4 a_{k,j} = \sum_{i=1}^4 a_{i,l} = 1, k,l = 1,2,3,4\right\}.$$
We then view $\cl P$ as a subset of $\cl M_4$ in a natural way
by identifying the family
$(p_{a,b|x,y})_{a,b,x,y}$ with the matrix
$$
\begin{bmatrix} p_{0,0|0,0} & p_{0,1|0,0} & p_{0,0|0,1} & p_{0,1|0,1} \\
p_{1,0|0,0} & p_{1,1|0,0} & p_{1,0|0,1} & p_{1,1|0,1}\\
p_{0,0|1,0} & p_{0,1|1,0} & p_{0,0|1,1} & p_{0,1|1,1}\\
p_{1,0|1,0} & p_{1,1|1,0} & p_{1,0|1,1} & p_{1,1|1,1}
\end{bmatrix}.$$
Under this identification, $\cl P$ is a convex subset of $BS_4.$

To facilitate numbering, we re-label the generators of $NC(2)$ as $E_1=p_1, E_2=q_1, E_3= p_2,$ and $E_4= q_2,$ 
where $p_i = \frac{1 + u_i}{2}$ and $q_i = \frac{1 - u_i}{2}$, $i = 1,2$. 
We note that $E_i$, $i = 1,2,3,4$, are positive operators satisfying the relations $E_1+E_2= E_3+E_4 = I.$ 
Similarly, we label the generators of the commutative operator system
$C(2) = {\rm Span} \{1,x_1,x_2 \}$ by $f_1 = (1 +x_1)/2, f_2= (1-x_1)/2, f_3 = (1+x_2)/2,$ and $f_4= (1-x_2)/2$. 


We recall from \cite{paulsen--todorov--tomforde2011} that every ordered $*$-vector space $W$ with positive cone $W_+$
and Archimedean unit $e$
can be equipped with a \lq\lq maximal'' operator system structure, that is, there exists a family $(C_n)_{n\in \mathbb{N}}$
of matrix cones such that $C_1 = W_+$ and $OMAX(W) = (W,(C_n)_{n\in \mathbb{N}}, e)$ is an operator system having the property that
every positive map $\phi : W\rightarrow \cl B(H)$ is completely positive when as a map from $OMAX(W)$ into $\cl B(H$). Similarly, there is a minimal operator system structure $OMIN(W)$ and, 
if $W$ is a finite-dimensional operator system, then $OMAX(W^d) = OMIN(W)^d$.

The following theorem summarizes some of our results and can be thought of as a dictionary between  the 
language of correlation boxes and that of tensor products of group operator systems.

\begin{theorem} We have the following identities:
\begin{multline*}
{\rm (i) \ \ } \cl P = (\cl V \omin \cl V)_+ \cap BS_4 = \\ \{ \left( s(E_i \otimes E_j) \right) \ : \  s \mbox{ is a state on } NC(2) \omax NC(2) \}; \end{multline*}
\begin{multline*} 
{\rm (ii) \ } \cl Q = (\cl V \omax \cl V)_+ \cap BS_4 =\\ \{ \left( s(E_i \otimes E_j) \right) \ : \  s \mbox{ is a state on }  NC(2) \omin NC(2)  \}; \end{multline*}
\begin{multline*} 
{\rm (iii) \ } \cl L = ( OMAX(\cl V) \omax OMAX(\cl V) )_+ \cap BS_4 =\\ \{ \left( s(f_i \otimes f_j) \right) \ : \  s \mbox{ is a state on } C(2) \omin C(2) \}. \end{multline*}
\end{theorem}
\begin{proof} 
(i) Since $\cl V \omin \cl V \coisubset \ell^{\infty}_4 \omin \ell^{\infty}_4 \equiv \cl M_4$, 
we have that $(\cl V \omin \cl V)_+$ is the set of matrices with non-negative entries satisfying $q_{i,1} + q_{i,2} = q_{i,3} + q_{i,4}$ and $q_{1,j} +q_{2,j} = q_{3,j} +q_{4,j}$,
 $i,j = 1,2,3,4$. From this it follows that $\cl P = (\cl V \omin \cl V)_+ \cap BS_4.$ But by Proposition~\ref{NC(2)dual} 
 and the fact that the maximal and the minimal tensor products are dual to each other, 
 we have that $(q_{i,j}) \in (\cl V \omin \cl V)_+$ if and only if $q_{i,j} = f(E_i \otimes E_j)$ for some positive functional $f: NC(2) \omax NC(2) \to \bb C$ and the first set of equalities follows.

(ii) Let $\cl Q_1=\{ \left( s(E_i \otimes E_j) \right) : s \mbox{ is a state on }  NC(2) \omin NC(2)  \}$, and  
let $A_i$ and $B_i$, $i = 1,2,3,4$, be positive operators acting on Hilbert spaces $\cl H$ and $\cl K$, 
respectively, 
such that $A_1 + A_2 = A_3 + A_4 = I$ and $B_1 + B_2 = B_3 + B_4 = I$, and $\rho$ be a
positive operator of trace class on $\cl H\otimes\cl K$. 
The maps $\phi : NC(2)\to \cl B(\cl H)$ and $\psi : NC(2)\to \cl B(\cl K)$ given by $\phi(E_i) = A_i$
and $\psi(E_i) = B_i$, $i = 1,2,3,4$, are unital and completely positive; hence, the tensor product map 
$\phi \otimes\psi : NC(2)\otimes_{\min} NC(2) \to \cl B(\cl H\otimes\cl K)$ is a unital completely positive map. 
Thus, the linear functional $\omega : NC(2)\otimes_{\min} NC(2)$ given by 
$\omega(u) = Tr(\rho(\phi\otimes\psi)(u))$, $u\in NC(2)\otimes_{\min} NC(2)$, is a state. 
It follows that all quantum correlation boxes are contained in $\cl Q_1$; 
since $\cl Q_1$ is a closed set, we conclude that $\cl Q\subseteq\cl Q_1$. 

To show that $\cl Q_1\subseteq \cl Q$, suppose that $\omega$ is a state of $NC(2)\otimes_{\min} NC(2)$.
Represent the C$^*$-algebra $\cstar(\mathbb{Z}_2\ast\mathbb{Z}_2)$ faithfully on a Hilbert space $\cl H$.  
Then $\omega$ has an extension to a state (denoted in the same way) 
of $\cstar(\mathbb{Z}_2\ast\mathbb{Z}_2)\otimes_{\min} \cstar(\mathbb{Z}_2\ast\mathbb{Z}_2)\subseteq \cl B(\cl H\otimes\cl H)$.
Since the latter C$^*$-algebra is separable, 
$\omega$ can be  approximated pointwise by restrictions of normal states on $\cl B(\cl H\otimes\cl H)$. 
We may therefore assume that $\omega$ is the restriction to 
$\cstar(\mathbb{Z}_2\ast\mathbb{Z}_2)\otimes_{\min} \cstar(\mathbb{Z}_2\ast\mathbb{Z}_2)$ of a 
normal state on $\cl B(\cl H\otimes\cl H)$. Hence, there exists a positive trace class operator $\rho$
on $\cl H\otimes\cl H$ such that $\omega(u) = Tr(u\rho)$, $u\in NC(2)\otimes_{\min} NC(2)$. 
Letting $\iota : NC(2)\to \cstar(\mathbb{Z}_2\ast\mathbb{Z}_2)$
be the canonical inclusion, we see that
$\omega(E_i\otimes E_j) = Tr((\iota(E_i)\otimes\iota(E_j))\rho)$, $i,j = 1,2,3,4$. 
Thus, $(\omega(E_i\otimes E_j))\in \cl Q$, and hence $\cl Q = \cl Q_1$. 

To show that $\cl Q$ coincides with $(\cl V \omax \cl V)_+ \cap BS_4$, we need only 
take into account that $(NC(2) \omin NC(2))^d = \cl V \omax \cl V$.

(iii) Let $\cl L_1=\{ \left( s(f_i \otimes f_j) \right): s \mbox{ is a state on } C(2) \omin C(2) \}$, and
suppose that we are given a local box whose entries have the form 
$$p_{a,b|x,y} = \sum_{\lambda} r(\lambda) p^1_{a|x}(\lambda) p^2_{b|y}(\lambda), \ \ \ a,b,x,y\in \{0,1\}.$$
Letting $s_{\lambda}^1$ and $s_{\lambda}^2$ be the states on $C(2)$ given by
$$(s_{\lambda}^k(f_1),s_{\lambda}^k(f_2),s_{\lambda}^k(f_3),s_{\lambda}^k(f_4)) = 
(p^k_{0|0}(\lambda),p^k_{1|0}(\lambda),p^k_{0|1}(\lambda),p^k_{1|1}(\lambda)), \ \ k = 1,2,$$
and $s = \sum_{\lambda} r(\lambda) s_{\lambda}^1 \otimes s_{\lambda}^2$, we have that 
$(p_{a,b|x,y})$ coincides with $(s(f_i\otimes f_j))$. It follows that $\cl L\subseteq \cl L_1$. 
Conversely, if 
$s$ is any state on $C(2) \omin C(2)$, then $s$ can be extended to a state on 
$C([-1,1]) \otimes_{\min} C([-1,1])$ and is hence given by 
integration against a probability measure on $[-1,1]\times [-1,1]$. 
But every such measure is in the weak* closed convex hull of the point evaluations, and it follows that 
$\cl L_1\subseteq \cl L$; thus, $\cl L = \cl L_1$. 

Now $(C(2) \omin C(2))^d = C(2)^d \omax C(2)^d.$ 
Since $C(2)= OMIN(C(2))$, we have that $C(2)^d = OMAX(C(2)^d).$ Finally, 
the equality $C(2)^d = \cl V$ of ordered spaces implies the equality 
$C(2)^d = OMAX(\cl V)$ as operator systems; hence, the proof of (iii) is complete.
\end{proof}

Now that we see the identifications of images of states with sets of boxes we may apply Tsirelson's result \cite{tsirelson1993} that
$\cl Q \ne \cl P$ to deduce directly that $[NC(2) \omin NC(2)]_+ \ne [NC(2) \omax NC(2)]_+.$

Now that we see the connection between sets of boxes and tensor products more clearly it is interesting to ask the following questions:

\begin{question} For each tensor product $\tau$ between $\min$ and $\max$, we obtain a set of boxes between $\cl Q$ and $\cl P$ by considering the convex set
\[ (\cl V \otimes_{\tau} \cl V)_+ \cap BS_4.\]
What are the relationships among these for the tensor products that have already
been introduced? Do any of these have physical meaning?
\end{question}

Dually, we could generate boxes by considering tensor products on $NC(2) \otimes NC(2).$ 
But since $NC(2) \omin NC(2) = NC(2) \oc NC(2)$, we would only get new families of boxes by looking at tensors between $\oc$ and $\omax$. 
The tensors that have been introduced lie between $\omin$ and $\oc$, and so they generate no new families of boxes when applied to $NC(2).$

\begin{question} In a similar fashion one could look at tensors on either $C(2) \otimes C(2)$ or its dual to generate new families of boxes between $\cl L$ and $\cl P.$ What are the relationships here and do any have physical meaning?
\end{question}


\end{document}